\documentclass{amsart}
\usepackage[fontsize=12pt]{scrextend}
\usepackage{amsmath}
\usepackage{mathtools}
\usepackage{amscd}
\usepackage{graphics}
\usepackage{latexsym}
\usepackage{color}
\usepackage{verbatim}
\usepackage{extarrows}

\usepackage{amssymb}
\usepackage{latexsym}
\usepackage{enumerate}
\usepackage{arydshln}
\usepackage{soul}
\usepackage{tikz,pgfplots}
\usepackage{tikz-cd}
\usepackage{ulem} 
\definecolor{darkgreen}{HTML}{3CB50F}
\usetikzlibrary{arrows,calc,backgrounds}

\usepackage[margin=1in]{geometry}
\usepackage[all]{xy}
\theoremstyle{plain}
\newtheorem{theorem}{Theorem}[section]

\newtheorem{cor}[theorem]{Corollary}
\newtheorem{lem}[theorem]{Lemma}
\newtheorem{prop}[theorem]{Proposition}

\theoremstyle{definition}
\newtheorem{definition}[theorem]{Definition}

\newtheorem{conj}[theorem]{Conjecture}
\newtheorem{remark}[theorem]{Remark}

\newtheorem{rmk}[theorem]{Remark}

\newtheorem{example}[theorem]{Example}

\theoremstyle{plain}
\newtheorem*{thm1}{Theorem A}
\newtheorem*{thm2}{Theorem B}
\newtheorem*{coro}{Corollary \ref{cor: EffRigid}}

\theoremstyle{remark}

\newcommand{\PP}{\mathbb{P}}
\newcommand{\EFF}{\mathrm{Eff}}

\newcommand{\Mov}{\mathrm{Mov}}

\newcommand{\MOV}{\mathrm{Mov}}

\newcommand{\OO}{\mathcal O}
\newcommand{\OOT}{\mathcal{O}_{\mathbb{P}^2}}

\newcommand{\Hom}{\operatorname{Hom}}

\def\Hom{{\rm Hom}\,}

\def\PP{{\mathbb P}}

\begin{document}

\title{Geometry of syzygies of sheaves on $\mathbb{P}^2$ via interpolation and Bridgeland stability}

\author{Manuel Leal}
\author{C\'esar Lozano Huerta}
\author{Tim Ryan}

\address{Currently: Harvard University\\
Department of Mathematics \\
Oxford 1, Cambridge, MA, USA.}
\email{lozano@math.harvard.edu}

\address{Permanent: Universidad Nacional Aut\'onoma de M\'exico\\
Instituto de Matem\'aticas, Unidad Oaxaca \\ Mex.}
\email{lozano@im.unam.mx}

\address{Universidad Nacional Aut\'onoma de M\'exico\\
Instituto de Matem\'aticas, Unidad Oaxaca \\
Oaxaca, Mex.}
\email{maz.leal.camacho@gmail.com}

\address{North Dakota State University\\
Department of Mathematics \\
Fargo, USA.}
\email{timothy.ryan.3@ndsu.edu} 


\begin{abstract} 
We show that the minimal free resolution of a general semi-stable sheaf $U$ on $\PP^2$ contains a subcomplex that determines an extremal ray of the cone of effective divisors of its moduli space. We provide evidence that this is part of a general phenomenon in which minimal free resolutions, for distinct Betti tables, contain subcomplexes depending on wall-crossing.  From this viewpoint, we provide new computations of the movable cones and Mori decompositions of some moduli spaces of sheaves using syzygies.
\end{abstract}

\maketitle
\noindent

\section{Introduction}
\label{sec: intro}
\noindent
A classical way to study a vector bundle $U$, or an ideal sheaf of points, on $\PP^2$ is by looking at its generators and their relations. These yield a free resolution which, following \cite{E}, can be written as follows
\[0 \xrightarrow{\hspace{7mm}} \bigoplus_{i=1}^m \mathcal{O}_{\PP^2}(-d_{2,i})\overset{M}{\xrightarrow{\hspace{7mm}}} \bigoplus_{i=1}^{m+k} \mathcal{O}_{\PP^2}(-d_{1,i}) \xrightarrow{\hspace{7mm}} U \xrightarrow{\hspace{7mm}} 0.\]
This (bulky) notation tells us that $U$ has $m+k$ generators (each of degree $d_{1,i}$) which satisfy $m$ relations (each in degree $d_{2,i}$ and  recorded in the matrix $M$) called \textit{syzygies}. These resolutions and syzygies are our main characters in this paper.

\medskip\noindent
A free resolution of $U$ can be made minimal by considering a minimal set of generators and a minimal set of relations among them and is unique up to isomorphism. Then, this minimal free resolution should not only tell us something important about $U$, but also about its moduli space. However, how it may do so has remained unclear except in special cases despite extensive study, 
and the present paper aims to clarify this situation.

\medskip\noindent
We show that the matrix $M$ in the minimal free resolution of $U$ after an \textit{appropriate} change of basis, whose exact description is in general subtle, can be written as
\[M=\left(
\begin{array}{c|c}
0  & A \\ \hline
 B & C 
\end{array}\right),\]
and the purpose of this paper is to exhibit the following correspondence
\begin{center}
\begin{tabular}{ c c c}
 $\left\{\begin{tabular}{ c}
 Properties of \\ $B$
 \end{tabular}\right\}$  & $\longleftrightarrow$ & $\left\{\begin{tabular}{ c}
 Bridgeland walls \& Base locus walls \\ of the moduli space of $U$
 \end{tabular}\right\}$
\end{tabular}
\end{center}

\medskip\noindent
Let us make this precise. If $U$ is a Mumford semi-stable sheaf on $\PP^2$, with character $\xi$, then it defines a point in its moduli space, denoted $M(\xi)$. This space has desirable properties: it is log Fano, $\mathbb{Q}$-factorial of expected dimension, and it has Picard rank $\le 2$. Then, we want to further understand: (1) The cone of effective divisors $\EFF \ \! M(\xi)$ and (2) The stable base locus decomposition (SBLD) of $\EFF \ \! M(\xi)$. Item (1) is answered in \cite{CHW} using interpolation. In \cite{LZ}, it is shown that walls of the SBLD are equivalent to Bridgeland walls in the space of stability conditions and thus the two items above can be understood using Bridgeland stability. This however requires substantial machinery, and our goal in this paper is to exhibit how syzygies, which are classical objects, shed light on (1) and (2). That is, we investigate:

\medskip
\textbf{Question 1:} \textit{How do minimal free resolutions determine the SBLD or Bridgeland walls?}

\medskip\noindent
Here is an example of what the answer to this question looks like. Consider $M(\xi)=\PP^{2[7]}$ the Hilbert scheme of 7 points on $\PP^2$, and let us focus on Item (1). In this case, the extremal innermost Bridgeland wall of $\mathbb{P}^{2[7]}$, called collapsing wall, consists of the stability conditions that make $\mathcal{T}_{\PP^2}(-4)$ a destabilizing object of a general $\mathcal{I}_Z$ \cite[\S 9]{ABCH}. Also, the nontrivial extremal wall of $\EFF\ \! \PP^{2[7]}$ is generated by the divisor $$J:= \{Z\in \PP^{2[7]}\ | \ h^0(E_{12/5}\otimes \mathcal{I}_Z)\ne 0 \},$$
where $E_{12/5}$ stands for \textit{the} exceptional bundle of slope $12/5$ \cite{CHW, Hui14}. 

\medskip\noindent
Let's see how we get the collapsing Bridgeland wall and the divisor $J$ from the minimal free resolution. If $Z\in \PP^{2[7]}$ has general Betti table, then its ideal sheaf $\mathcal{I}_Z$ is generated by 3 cubics $c_1,c_2,c_3$ with syzygies  $l_1c_1+l_2c_2+l_3c_3=0$ and $q_1c_1+q_2c_2+q_3c_3=0$,
where $l_i,q_i$ are linear and quadratic forms, respectively. These forms yield the map in the minimal free resolution of 
$\mathcal{I}_Z$: 
\begin{align}\label{R1}
0\xrightarrow{\hspace{7mm}} \OO(-4)\oplus\OO(-5)\overset{M=\left(
\begin{array}{cc}
l_1  & q_1 \\ 
 l_2 & q_2 \\
  l_3 & q_3 \\
\end{array}\right)}{\xrightarrow{\hspace{18mm}}} \OO(-3)^3\xrightarrow{\hspace{7mm}} \mathcal{I}_Z\xrightarrow{\hspace{7mm}} 0.
\end{align}
Let $B$ be the matrix formed by the linear forms $\{l_1,l_2,l_3\}$. When these forms are linearly independent, $B$ yields the minimal free resolution of its cokernel $\mathcal{T}_{\PP^2}(-4)$. Then, (\ref{R1}) gives
\begin{align}\label{R2}
\OO(-5)\hookrightarrow \mathcal{T}_{\PP^2}(-4)\twoheadrightarrow \mathcal{I}_Z.
\end{align}

\medskip\noindent
Thus, Resolution (\ref{R1}) \textit{contains} the minimal free resolution of the destabilizing object $\mathcal{T}_{\PP^2}(-4)$ that yields the collapsing wall of $\PP^{2[7]}$. Moreover, Resolution (\ref{R2}) implies that $Z$ is not in $J$. On the other hand, if the matrix $B$ fails to resolve $\mathcal{T}_{\PP^2}(-4)$, then $Z\in J$.

\medskip\noindent
We carry out this analysis in full generality to address Item (1). We focus (like in the example above) on the most natural extremal chamber, called primary.

\begin{thm1}[Theorem \ref{MAINdetailed}] \label{thm: main}
A general sheaf in $M(\xi)$ is not in the stable base locus of the primary extremal chamber  of the effective cone $\EFF \ \! M(\xi)$ 
if and only if its minimal free resolution contains \footnote{See Section \ref{preliminaries} for the formal definition.} the minimal free resolution of the general Bridgeland destabilizing object which induces the extremal wall of $\EFF\ \!M(\xi)$.
\end{thm1} 

\medskip\noindent
Theorem A identifies \textit{the} subcomplex of the minimal free resolution of a general sheaf that determines an extremal divisor $J$ in $\EFF \ \! M(\xi)$.  Indeed, Corollary \ref{cor: EffRigid} writes down the minimal free resolution of sheaves in $J$, when $J$ is rigid; or in its stable base locus when $J$ is movable.

\medskip\noindent
As for other walls of the SBLD, it is clear that syzygies carry information of some of them, $e.g.$, those near the ample cone \cite{CH13}, [ABCH, Sec. 10],\cite[p. 23]{Hui14}. But, how exactly they do so in general has been obscure. However, a closer look at Theorem A reveals the following pattern which allows us to expand its conclusions.

\medskip\noindent
In Theorem A, a general sheaf is not in the base locus of its moduli space $M(\xi)$ if and only if its minimal free resolution is a mapping cone of two distinguished subcomplexes:  those giving rise to a Bridgeland destabilizing sequence that induces the collapsing wall. That is, when this fails, then a sheaf \textit{enters} the base locus.  For example, consider a general $Z\in \mathbb{P}^{2[7]}$, then Resolution (\ref{R1}) is the mapping cone of the map of complexes $W_{\bullet}\to F_{\bullet}$ as follows  
\begin{center}
    \begin{tikzcd}
        W_{\bullet}\!\! \!\!\! \arrow{d} \quad   &  & \mathcal{O}_{\PP^2}(-5)\arrow{d}{C} &  \\  
       F_{\bullet} \!\! \!\!\! \quad   & \mathcal{O}_{\PP^2}(-4) \arrow{r}{B}  & \mathcal{O}_{\PP^2}(-3)^3   & . 
    \end{tikzcd}
\end{center} 
where the maps $B,C$ are defined by the linear and quadratic forms, respectively. If Resolution (\ref{R1}) fails to determine Resolution (\ref{R2}), which is a destabilizing sequence, then $Z$ is in the base locus of the nontrivial extremal chamber of $\EFF \ \! \PP^{2[7]}$. This exemplifies the notion that seems to relate syzygies to walls in general. We provide evidence of this claim below.

\medskip\noindent
Consider $M(\xi)\cong \PP^{2[n]}$, the Hilbert scheme of $n$ points on $\PP^2$. Let us focus on the values $n$, for which the nontrivial extremal divisor $J\in \EFF \ \! \PP^{2[n]}$ is determined by Betti numbers ---unlike $n=7$. They are $n=\tfrac{1}{2}d(d+1)$ or $n=2d(d+1)$, with $d>0$. Let us then describe the next wall. This is the first computation of the movable cone $\MOV(\PP^{2[n]})$ using syzygies.

\medskip\noindent
\begin{thm2}[Theorem \ref{theoremB}] 
A general\footnote{any $Z$ with a divisorial Betti table, see Section \ref{sec: mov}.}
sheaf $Z$ in $J\subset \PP^{2[n]}$  enters the base locus at 
the nontrivial extremal wall of the movable cone $\MOV(\PP^{2[n]})$. 
Also, it contains the minimal free resolution of a Bridgeland destabilizing object which induces such an extremal wall of $\MOV(\PP^{2[n]})$.
\end{thm2}

\medskip\noindent
We prove Theorem B using interpolation of vector bundles and syzygies, and exhibit that the Bridgeland destabilizing object arises. Here, the pattern dictated by Theorem A holds: a minimal free resolution, this time of the general element in $J$, is the mapping cone of Bridgeland destabilizing objects inducing the wall after the collapsing one. 

\medskip\noindent
If we want to prove that this pattern holds in general, then we must show that syzygies characterize walls of the SBLD, and verify they do so precisely when they are mapping cones that give rise to destabilizing sequences and walls.
To do this, we propose the characterization of walls in the SBLD using syzygies and interpolation as the \textit{base locus decomposition program}. We run this program for several cases (Sections 4, 5) and notice that the pattern described above holds: syzygies characterize walls of the SBLD precisely when they form mapping cones of Bridgeland destabilizing objects. This reveals the geometry of syzygies via wall-crossing; and also answers Question 1 completely.  

\begin{center} 
   \includegraphics[scale=.15]{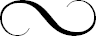} 
\end{center}

\subsection{Detailed results}
\noindent
We now describe our results in more detail and outline their proofs. We study our Question above by associating to a sheaf $U$, the (potential) wall in the SBLD where it enters the base locus. Let us describe how this works. The moduli space $M(\xi)$ is always assumed to have Picard rank 2, which is true for any $\xi$ above the Dr\'ezet-Le Potier curve \cite{DL}.

\medskip\noindent
\textbf{Interpolation problem for vector bundles \cite[p. 2]{CHW}:} Given a sheaf $U\in M(\xi)$, find a semi-stable bundle $V$ of minimal slope such that
\begin{enumerate}
\item[(a)] $\mu(V)\in \mathbb{Q}$,
\item[(b)] $\mu(\xi)+\mu(V)\geq 0$, and 
\item[(c)] $U\otimes V$ has no cohomology.
\end{enumerate}

\medskip\noindent
If the interpolation problem above is solved for a given $U\in M(\xi)$, then the Brill-Noether divisor $$D_V:=\{F\in M(\xi)\ | \ h^1(F\otimes V)\ne 0\} $$ is an effective divisor of $M(\xi)$ which does not have $U$ in its stable base locus. Thus, the class $D_V$ spans the ray $r_V\subset$ $\mathrm{NS}(M(\xi))_{\mathbb{Q}}\cong\mathbb{Q}^2$ at which $U$ potentially enters the base locus: the divisor classes on one side of $r_V$ do not contain $U$ in their stable base locus and, because of the minimality of $\mu(V)$, potentially do contain it on the other side of $r_V$. That makes $r_V$ a candidate for a wall of the SBLD.  

\medskip\noindent
Coskun et al. found the (primary) extremal wall of $\EFF(M(\xi))$ by solving the interpolation problem for a general $U\in M(\xi)$ \cite[\S 5]{CHW}. They did so by showing that $U$ admits a resolution in terms of very particular exceptional bundles; the so-called \textit{generalized Gaeta resolution} of $U$. Building on this, it suffices to exhibit how the minimal free resolution of $U$ determines the generalized Gaeta resolution. 

\medskip\noindent
Here is a brief description of the generalized Gaeta resolution. There is an exceptional bundle canonically associated to the moduli space $M(\xi)$: the controlling exceptional bundle $E_{(\alpha\cdot \beta)}$ (Definition \ref{Def:controlling}). Let $e$ denote its Chern character. 
Now, the Beilinson spectral sequence yields a short exact sequences of sheaves, the generalized Gaeta resolution, that depends on the sign of the relative Euler characteristic $\chi(\xi,e)$ as follows:
\begin{equation}\label{gaetatrianglepositive}
0\xrightarrow{\hspace{7mm}} E^k_{-\alpha-3}\xrightarrow{\hspace{7mm}} E^l_{-\beta}\oplus E^m_{-(\alpha\cdot\beta)}\xrightarrow{\hspace{7mm}} U\xrightarrow{\hspace{7mm}} 0 \quad \text{if }\chi(\xi,e)>0,
\end{equation}
\begin{equation}\label{gaetatrianglenegative}
      0 \xrightarrow{\hspace{7mm}}  E_{-(\alpha\cdot\beta)-3}^{m}\oplus E_{-\alpha-3}^k \xrightarrow{\hspace{7mm}} E^l_{-\beta} \xrightarrow{\hspace{7mm}} U \xrightarrow{\hspace{7mm}} 0  \quad \text{if }\chi(\xi,e)<0,\text{ and}
\end{equation}
\begin{equation}\label{exceptionalgaetatriangle}
   0\xrightarrow{\hspace{7mm}} E^k_{-\alpha-3}\xrightarrow{\hspace{7mm}} E^l_{-\beta} \xrightarrow{\hspace{7mm}} U\xrightarrow{\hspace{7mm}} 0  \quad \text{if }\chi(\xi,e)=0.
\end{equation}
Here, $k$, $l$ and $m$ are determined by the spectral sequence (see Sec. \ref{subsubsec: GaetaTriangle}). 

\medskip\noindent
On the other hand, the minimal free resolution of a sheaf $U\in M(\xi)$ can be viewed as a 2-term complex of line bundles $U_\bullet$ whose only non-zero cohomology is $U$.  To relate this complex to the generalized Gaeta resolution, we need to understand its subcomplexes using the following definition.

\begin{definition}\label{generalgaeta}
A 2-term complex  $U_{\bullet}$ of line bundles is called 
\begin{itemize}
\item $F$-admissible if it is the mapping cone of a map to $F_{\bullet}$, the minimal free resolution of a torsion-free sheaf $F$,  from some other complex,
\item $W$-residual if it is the mapping cone of a map from the dual of the minimal free resolution of the sheaf $W^*=\mathcal{H}om(W,\mathcal{O}_{\PP^2})$ to some other complex, 
\item a $(F,W)$-cone if it is the mapping cone of a map to the minimal free resolution of a sheaf $F$ from the dual of the minimal free resolution of the sheaf $W^*$.
\end{itemize}
\end{definition}

\medskip\noindent
In Section \ref{section2.4}, we elaborate on this definition, but for the time being let us comment that \textit{mapping cone} in our context carries the following significance. A sheaf $U$ is $F$-admissible if and only if its minimal free resolution 
\begin{equation*}
0\xrightarrow{\hspace{7mm}} L_1\overset{M}{\xrightarrow{\hspace{7mm}}} L_0\xrightarrow{\hspace{7mm}}U \xrightarrow{\hspace{7mm}} 0
\end{equation*}
satisfies that the map $M$ has a matrix representation, up to row and column reduction, as
\[M=\left(
\begin{array}{c|c}
0  & A \\ \hline
 B & C 
\end{array}\right),\]
where the submatrix $B$ is the map in the minimal free resolution of the sheaf $F$. Thus, $M$ induces a map $F\xrightarrow{\hspace{3mm}} U$. The complement of the subcomplex determined by $B$, denoted $W$, is called \textit{residual complex} and it exhibits $U$ as $W$-residual. 
Using the subindices from the previous page, let us restate Theorem A:

\begin{theorem}\label{MAINdetailed} 
Let $U\in M(\xi)$ be a sheaf and $E_{\alpha\cdot \beta}$ denote the controlling exceptional bundle of $\xi$. If $U_{\bullet}$ stands for the minimal free resolution of $U$, then:
\begin{enumerate}
\item[(a)] $U$ admits Resolution \eqref{gaetatrianglepositive} if and only if $U_{\bullet}$ is $E_{-(\alpha \cdot \beta)}^m$-admissible and the residual complex is a $(E_{-\beta}^l, E_{-\alpha - 3}^k)$-cone.

\medskip
\item[(b)] 
$U$ admits Resolution \eqref{gaetatrianglenegative} if and only if $U_{\bullet}$ is $E_{-\alpha.\beta-3}^m$-residual and the residual complex is an $(E_{-\beta}^{l},E_{-\alpha-3}^{k})$-cone.

\medskip
\item[(c)] $U$ admits Resolution \eqref{exceptionalgaetatriangle} if and only if $U_{\bullet}$ is a $(E_{-\beta}^{l},E_{-\alpha-3}^{k})$-cone.
\end{enumerate}

In all cases, the minimal free resolution of $U$ is the mapping cone of the map to the destabilizing object from (a shift of) the quotient that yield the collapsing Bridgeland wall.
\end{theorem}

\medskip\noindent
The strategy to prove Theorem \ref{MAINdetailed} is straightforward. Suppose $U$ fits into a generalized Gaeta resolution, e.g., in Case (c),
\[0\xrightarrow{\hspace{7mm}} W \overset{f}{\xrightarrow{\hspace{7mm}}}  F\xrightarrow{\hspace{7mm}} U\xrightarrow{\hspace{7mm}} 0.\] 
In $D^b(\PP^2)$, we may replace $W$ and $F$ by equivalent complexes of line bundles induced by minimal free resolutions. Then, $f$ lifts to a map of such complexes and its mapping cone is a complex which only involves line bundles. We then argue that this mapping cone is the minimal free resolution of $U$. As a consequence of this theorem we have the following. 

\begin{coro}
If a general $U\in M(\xi)$ fits into Resolution \eqref{exceptionalgaetatriangle}, then, the closure in $M(\xi)$ of the locus of sheaves whose minimal free resolution fail to be $(E_{-\beta}^l, E_{-\alpha - 3}^k)$-cones forms an irreducible and reduced divisor. This divisor is extremal in $\EFF(M(\xi))$, and it is the stable base locus of an extremal chamber of $\mathrm{Eff}(M(\xi))$.
\end{coro}

\medskip\noindent
We now investigate how the syzygies determine inner walls of the SBLD of $\EFF(M(\xi))$. 
Let $J\subset \PP^{2[n]}$ be the nontrivial extremal effective divisor.

\begin{theorem}\label{theoremB}
Solving the interpolation problem for a general $Z\in J$, we deduce that the movable cone of $\PP^{2[n]}$ is
\begin{align*}
    (a)\ \MOV(\PP^{2[n]}) \ &=\ \langle H,\tfrac{d^2 - 2 d + 2}{d - 1}H-\tfrac{1}{2}B\rangle\ \text{ if }\quad n=\tfrac{d(d+1)}{2} \text{ and}\\[2mm]
    (b)\ \MOV(\PP^{2[n]})\ &=\ \langle H, \tfrac{8 d^2 - 4 d + 1}{4 d - 1}H-\tfrac{1}{2}B\rangle\text{ if }\quad 
    n=2d(d+1).
\end{align*}
Moreover, the general $Z\in J$ is $F$-admissible, where $F$ is a destabilizing object that determines the nontrivial extremal ray of the movable cone $\MOV(\PP^{2[n]})$.
\end{theorem}

\medskip\noindent
We prove this theorem by writing the minimal free resolution of a general element $Z\in J$, using the previous corollary and solving the interpolation problem for it. Then, Theorem \ref{theoremB} simply observes that $(\mathcal{I}_Z)_{\bullet}$ is $F$-admissible, where $F$ is the Bridgeland destabilizing object giving rise to the extremal wall of $\MOV(\PP^{2[n]})$. Solving the interpolation problem in this case is difficult and has not been done before using these techniques. But despite the difficulty, observe that this process turns Theorem \ref{theoremB} into a particular case of the following:

\medskip\noindent 
\textbf{Base locus decomposition program:} In order to understand how syzygies determine walls of the stable base locus decomposition (SBLD) of $\EFF(M(\xi))$, we proceed as follows. Set $M = M(\xi)$. 
Then,
\begin{enumerate}
\item Solve the interpolation problem for a generic point in $M$ using the minimal free resolution. The outcome is a bundle $V$.
\item The Brill-Noether divisor $D_V$ spans a potential wall in the SBLD. 
\item The divisor $D_V$ has base locus of its own, so we may repeat the process from Step (1), replacing $M$ with the stable base locus of $D_V$.
\end{enumerate}

\medskip\noindent
If Step (2) yields a wall, then this program would answer Question 1 as long as we reach \textit{all} the walls in this way. 
If $M(\xi)=\PP^{2[n]}$, then this program computes the entire SBLD of $\EFF(\PP^{2[n]})$ for small values of $n$ (Section \ref{section5}), and it exhibits how syzygies characterize walls of the SBLD. We verify they do so precisely when they are mapping cones of Bridgeland semi-stable objects that give rise to walls. This suggests the following syzygy behavior:
\begin{conj}\label{conj}
A sheaf $U$ is $F$-admissible, where $F$ is a Bridgeland destabilizing object that yields the wall $W$ when $U$ enters the stable base locus at the wall $W$. \end{conj}

\medskip\noindent
\textbf{Relation to existing literature:} The stable base locus decomposition SBLD of $M(\xi)$ is determined by Bridgeland stability conditions \cite[Theorem B]{LZ}; but how they do so in terms of the minimal free resolutions is new. 
Here, thinking of the minimal free resolution of $U$ as a mapping cone clarifies the nature of the heart of the t-structure of the torsion pair $\mathcal{A}_s$ used in Bridgeland stability, as explained in \cite{ABCH}: $U$ is $F$-admissible for a torsion-free sheaf $F$ and the residual complex $W$ (which is given by some of the syzygies) often yields a sheaf. This sheaf has a clear geometric meaning in all the cases of this paper (e.g., see Remark \ref{folations} for a relation to the Campillo-Olivares problem on foliations \cite{CO,CO1,O}, or Remark \ref{residual1}). Further, the \textit{Positivity Lemma} by Bayer and Macr\`i \cite[Lemma 3.3]{BM} exhibits a nef divisor, and its dual curve, in the Bridgeland moduli space along a wall. Proposition \ref{Prop:: InterpoTang} gives a previously unknown description of this curve in classical terms which involves, in an unexpected way, the residual complex $W$.

\medskip\noindent
Our proposed stable base locus decomposition program follows closely the computation of walls in \cite{CHW, Hui14, CH13}. Our contribution is to propose a general and systematic way of computing the SBLD using those techniques and syzygies. For example, this recovers the content of the examples in \cite[\S 9]{ABCH} with different methods. 

\section*{organization of the paper} \noindent
Section \ref{preliminaries} contains preliminaries and the notation of the objects in birational geometry we aim to study using free resolutions.
In Section \ref{sec: eff}, we prove Theorem \ref{MAINdetailed}. 
In Section \ref{sec: mov}, we prove Theorem \ref{theoremB} which involves two cases: triangular and tangential numbers. 
In Section \ref{section5}, we show that we can fully answer our Question above for $\mathbb{P}^{2[n]}$ with $n=2,3,4,5,6,12$.

\section*{acknowledgments}
\noindent 
We thank Izzet Coskun, Benjamin Schmidt, Jorge Olivares and Claudia Reynoso for useful conversations. This work has been partly supported by CONAHCYT, grant ``Estancias sab\'aticas en el extranjero 2023'' No. I1200/311/2023. CLH thanks the Department of Mathematics at Harvard for its warm hospitality during his sabbatical visit; in particular, Joe Harris, Anand Patel and Mihnea Popa for their encouragement and generosity in sharing their ideas. 
CLH is currently a CONAHCYT Research Fellow in Mathematics, project No. 1036.
TR was partially supported by Karen Smith's Keeler and Fulton professorships.

\bigskip
\section{Preliminaries}\label{preliminaries}
\noindent
This section recalls the notions about sheaves on $\PP^2$ that are used throughout the paper.
We work over $\mathbb{C}$, and all sheaves are coherent.

\subsection{The logarithmic Chern character}

\noindent
Let $U$ and $V$ be  coherent sheaves on $\PP^2$ with $U$ having Chern character $(r,\mathrm{ch}_1,\mathrm{ch}_2)$ and $r>0$.
The \textit{slope} and \textit{discriminant} of $U$  are defined to be $\mu(U)=\frac{\mathrm{ch}_1}{r} \;\text{ and }\; \Delta(U) =\frac{1}{2}\mu(U)^2-\frac{\mathrm{ch}_2}{r}$, respectively.
The slope and discriminant are logarithmic Chern classes in the sense that \[\mu(U\otimes V) = \mu(U) +\mu(V) \qquad \text{ and }\qquad \; \Delta(U\otimes V) = \Delta(U) +\Delta(V).\] For the dual bundle $U^*$, we have $\mathrm{rk}(U^*) =\mathrm{rk}(U)$, $\mu(U^*) = -\mu(U)$, and $\Delta(U^*) = \Delta(U)$.
Since the Chern character $(r,\mathrm{ch}_1,\mathrm{ch}_2)$ determines the slope and discriminant, throughout this paper we will refer to $\mathrm{ch}(U)=(r,\mu,\Delta)$ as the \textit{log Chern character}, or simply the character, of a sheaf $U$ of positive rank.
These definitions extend to complexes by taking alternating sums of the Chern character.

\medskip\noindent
A coherent, torsion-free sheaf is called \textit{slope (semi-)stable} if all proper nontrivial subsheaves have smaller (or equal) slope. 

\subsection{Orthogonality notions}
\label{subsec: euler}
\medskip\noindent
By the Riemann-Roch formula, the \textit{Euler characteristic} of $U$ is 
\[ \chi(U) = \sum_{i=0}^2(-1)^i h^i\left(\mathbb{P}^2,U\right) = \mathrm{rk}(U)\left(\frac{1}{2}\left(\mu(U)+1\right)\left(\mu(U)+2\right)-\Delta(U)\right).\]

\medskip\noindent
Similarly, the \textit{relative Euler characteristic} of two sheaves $U$ and $V$ is 
\begin{align}\label{CHI} 
\begin{split}
\chi(U,V) &= \sum_{i=0}^2(-1)^i \mathrm{ext}^i\left(U,V\right) \\         &=\mathrm{rk}(U)\mathrm{rk}(V)\left(\frac{1}{2}\left(\mu(V)-\mu(U)+1\right)\left(\mu(V)-\mu(U)+2\right)-\Delta(U)-\Delta(V)\right).
\end{split}
\end{align}

\medskip\noindent
Since the Euler characteristic $\chi(U)$ depends only on the character of $U$, it induces a bilinear pairing $$(\xi,\zeta):=\chi(\xi^*,\zeta) = \chi(\xi \otimes \zeta)$$ in the Grothendieck group $K_0(\PP^2)\otimes \mathbb{R}\cong \mathbb{R}^3$; in particular, it can be applied to complexes.  

\medskip\noindent
Denote by $\zeta^\perp$ the set of log Chern characters $\xi$ such that $\chi(\xi\otimes \zeta)=0$. Given a fixed slope and discriminant, the rank does not affect this equality so $\zeta^\perp$ defines a parabola in the $(\mu,\Delta)$-plane. Hence, we sometimes abuse notation and  write $\chi(\xi\otimes \zeta)=0$, or $\chi(\xi^*, \zeta)=0$, when $\xi$ is only a point in the $(\mu,\Delta)$-plane.

\begin{definition}\label{orthogonal}
Two coherent sheaves $U$ and $V$ are called \textit{numerically orthogonal} if $\chi(U,V)=0$ and \textit{cohomologically orthogonal} if $h^i(\mathbb{P}^2,U\otimes V)=0$ for $i=0,1,2$.
\end{definition}

\medskip\noindent
Numerical orthogonality only depends on Chern characters and is easy to check. The following observation will be used in showing cohomological orthogonality in Sections 4 and 5: If $U$ fits in a short exact sequence (or a similar triangle in $D^b(\PP^2)$)\[0\xrightarrow{\hspace{7mm}} W\xrightarrow{\hspace{7mm}} F\xrightarrow{\hspace{7mm}} U\xrightarrow{\hspace{7mm}} 0,\]
where $rk(W)>0$ and $\mu(W)\neq\mu(F)$, then up to choosing $r$, there is a unique character $\zeta = (r,\mu,\Delta)$ which is numerically orthogonal to $F$ and $W$. That is, $(\zeta,\mathrm{ch}(W))=(\zeta,\mathrm{ch}(F))=0$ has a unique solution, which is
\begin{align*}
    \mu&=\frac{\Delta(F)-\Delta(W)}{\mu(F)-\mu(W)}-\frac{\mu(W)+\mu(F)+3}{2} \text{ and}\\[2mm]
    \Delta &= \frac{(\Delta(F)-\Delta(W))^2}{(2 (\mu(F) - \mu(W))^2} +\frac{(\mu(F)-\mu(W))^2  - 4(\Delta(F) +\Delta(W)) - 1}{8}.
\end{align*}

\medskip\noindent
If $rk(U)>0$, then the same is true replacing $W$ by $U$.  This process determines characters numerically orthogonal to $U$ depending on a resolution of $U$.

\subsection{Exceptional bundles}\label{Exceptional bundles}
The geometry of coherent semi-stable sheaves on $\PP^2$ is largely controlled by exceptional vector bundles. We refer the reader to \cite{LP,DL} for an introduction to this fascinating circle of ideas. Let us recall the notions that we need.

\medskip\noindent
A stable vector bundle $E$ is called \textit{exceptional} if $\mathrm{ext}^i(E,E) =0$ for $i>0$, and only has homotheties as endomorphisms, i.e., $\mathrm{hom}(E,E)=1$.
In particular, this implies that $\chi(E,E)=1$, and consequently by \eqref{CHI}, its discriminant is \[\Delta(E) = \frac{1}{2}\left(1-\frac{1}{\mathrm{rk}(E)^2}\right).\]
This equation implies that the discriminant of an exceptional bundle is less than $\frac{1}{2}$.
In fact, exceptional bundles are the only stable sheaves with discriminant less than $\tfrac{1}{2}$. In addition, an exceptional bundle with slope $\alpha$ is unique up to isomorphism, therefore we can refer to \textit{the} exceptional bundle $E_{\alpha}$ with slope $\alpha$. Line bundles $\mathcal{O}_{\PP^2}(d)$, as well as the tangent bundle $\mathcal{T}_{\PP^2}$, are examples of exceptional bundles.

\medskip\noindent
Exceptional bundles on the plane sit in \textit{exceptional collections}. These are triples $\{E_\alpha,E_{\alpha\cdot \beta},E_\beta\}$, where the only nonzero $\mathrm{Ext}$-groups between them are $\Hom(E_\alpha,E_{\alpha\cdot \beta})$, $\Hom(E_\alpha,E_\beta)$, and $\Hom(E_{\alpha\cdot \beta},E_\beta)$.

\medskip\noindent
Each exceptional bundle $E$ determines two parabolas in the $(\mu,\Delta)$-plane defined by $\chi(E,*)=0$ and $\chi(*,E)=0$.  The equation of each of these parabolas is determined by \eqref{CHI} above. We denote the intersection points of these curves with the line $\Delta = \frac{1}{2}$ in a unit neighborhood of $E$ by $(\mu_L,\frac{1}{2})$ and $(\mu_R,\frac{1}{2})$, respectively.

\begin{definition}\label{Cherncharacter} For any exceptional character $\xi=(r,\mu,\Delta)$, the  characters $(1, \mu_L,\frac{1}{2})$ and $(1,\mu_R,\frac{1}{2})$ are called the right and left endpoint characters of $\xi$, respectively.  
\end{definition}

\noindent
This numeric data is useful in computations on the $(\mu,\Delta)$-plane such as the following definition. 

\begin{definition}[\cite{CHW}, Definition 3.2]\label{Def:controlling}
    Let $\xi$ be a character. We will call an exceptional bundle $E$ the \textit{primary controlling exceptional} bundle of $\xi$, if the intersection point $(r,\mu,\tfrac{1}{2})$ of $\xi^{\perp}$ and $\Delta = \frac{1}{2}$ with larger $\mu$ 
    lies between the left and right endpoint characters of $E$. 
\end{definition}

\medskip\noindent
Figure 1 displays $(\mu_L,\frac{1}{2})$ and $(\mu_R,\frac{1}{2})$ for $\xi=(2,\frac{1}{2},\frac{3}{8})$.

\medskip\noindent
The primary controlling exceptional bundle is uniquely determined by $\xi$. It will be useful in the proof of Theorem \ref{MAINdetailed} to write the controlling exceptional bundle of $\xi$ as $E_{\alpha.\beta}$, where 
$\{E_\alpha,E_{\alpha\cdot\beta},E_\beta\}$ is the exceptional collection
used in \cite[Proposition 5.3]{CHW}

\begin{center}
\begin{figure}[htb]\label{FIG1}
\resizebox{.75\textwidth}{!}{\includegraphics{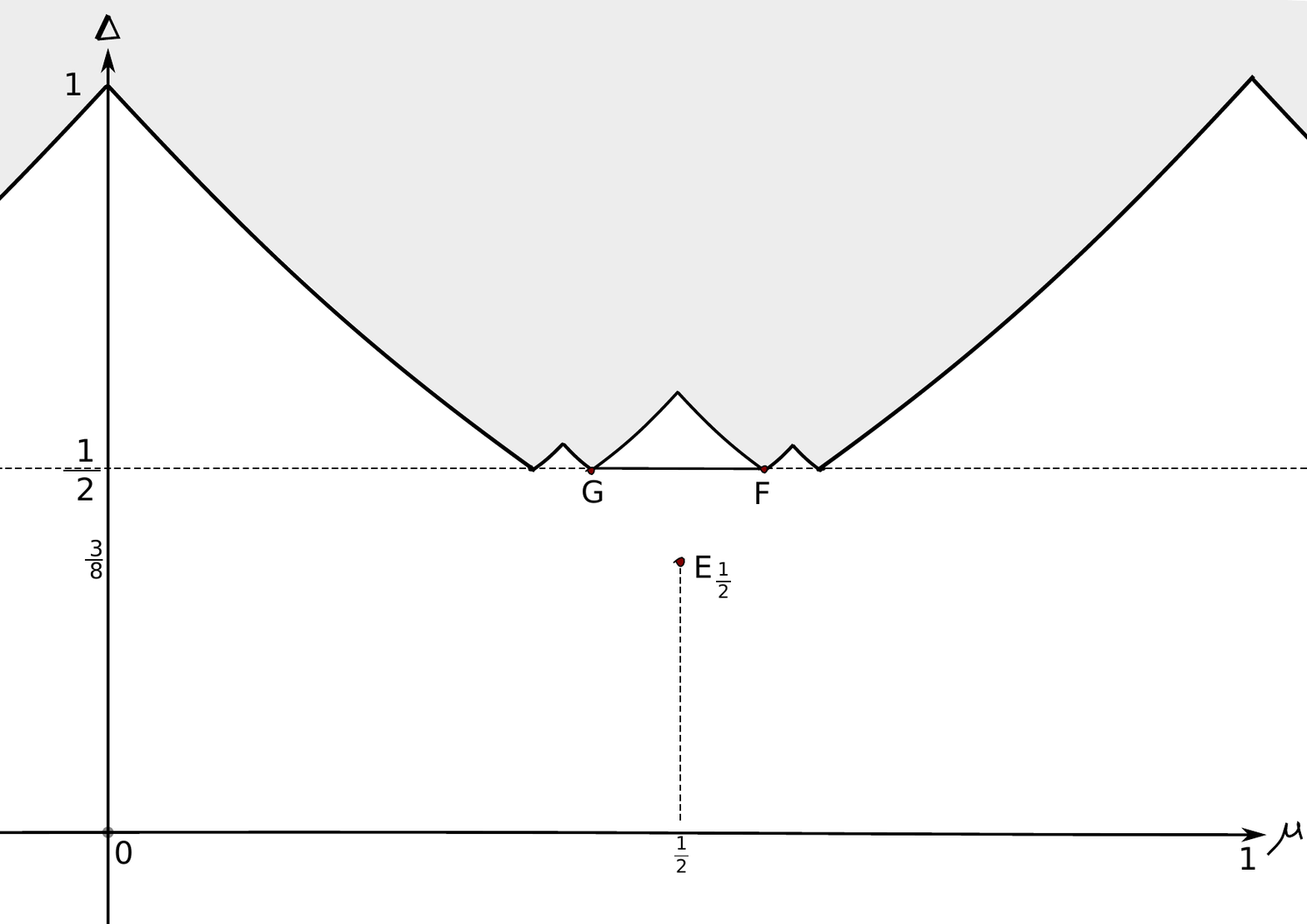}}
\caption{Endpoints $G=(\mu_L,\tfrac{1}{2})$, $F=(\mu_R,\tfrac{1}{2})$ prescribed by the Dr\'ezet-Le Potier curve over the exceptional bundle $E_{\frac{1}{2}}$.}
\end{figure}
\end{center}

\subsection{Minimal free resolution of a coherent sheaf on $\PP^2$}\label{section2.4}
The geometry of syzygies, and the minimal free resolutions they form, is a classical and robust topic of current research. We refer the reader to \cite{E} for a more complete introduction. We now recall what we need.

\medskip\noindent
The minimal free resolution of a coherent sheaf in projective space is a finer invariant than the Hilbert function and the Hilbert polynomial. It is unique up to isomorphism and for a torsion-free sheaf on the plane it is particularly simple: it depends only on 
generators and their relations. Indeed, any coherent torsion-free sheaf $U$ on $\mathbb{P}^2$ has a minimal free resolution of the form \[0 \xrightarrow{\hspace{7mm}} \bigoplus_{i=1}^m \mathcal{O}_{\PP^2}(-d_{2,i})\overset{M}{\xrightarrow{\hspace{7mm}}} \bigoplus_{i=1}^{m+k} \mathcal{O}_{\PP^2}(-d_{1,i}) \xrightarrow{\hspace{7mm}} U \xrightarrow{\hspace{7mm}} 0,\]
where $M$ is a matrix with homogeneous-polynomial entries.
The \textit{Betti number} $\beta_{1,j}$ is the number of summands 
such that $d_{1,i} = j$. 
In that same exact sequence, the number of summands 
such that $d_{2,i} = j+1$, is
the \textit{Betti number} $\beta_{2,j}$.

\medskip\noindent
Consider a two-term complex, such as the minimal free resolution of $U$ above,
\begin{equation}\label{2Cx}
  \mathcal{U}_\bullet: {\hspace{7mm}} \bigoplus_{i=1}^m \mathcal{O}_{\PP^2}(-d_{2,i})\overset{M}{\xrightarrow{\hspace{7mm}}} \bigoplus_{i=1}^{m+k} \mathcal{O}_{\PP^2}(-d_{1,i}). 
\end{equation}
If $\mathcal{S}_\bullet$ is a sub-complex of $\mathcal{U}_\bullet$,
$$\mathcal{S}_\bullet: \quad  \bigoplus_{i=1}^n \mathcal{O}_{\PP^2}(-d_{2,j_i})\overset{B}{\xrightarrow{\hspace{7mm}}} \bigoplus_{i=1}^{l} \mathcal{O}_{\PP^2}(-d_{1,\ell_i}),$$
with $n\le m$ and $l\le m+k$, then we will say that $\mathcal{S}_\bullet$ is \textit{contained} in $\mathcal{U}_\bullet$ if, up to row and column reduction, the matrix $B$ is a submatrix of $M$ and all the entries of $M$ above $B$ are $0$. Note that the indices $j_i$ and $\ell_i$ specify which minor of $M$ is $B$.

\medskip\noindent
If $\mathcal{S}_\bullet$ is contained in  $\mathcal{U}_\bullet$, then what is left over is the \textit{residual complex}:
$$\mathcal{R}_\bullet: \quad  \bigoplus_{i=n+1}^{m} \mathcal{O}_{\PP^2}(-d_{2,j_i})\overset{A}{\xrightarrow{\hspace{7mm}}} \bigoplus_{i=l+1}^{m+k} \mathcal{O}_{\PP^2}(-d_{1,\ell_i}).$$

\medskip\noindent
Observe that if a minimal free resolution $\mathcal{U}_{\bullet}$ is the mapping cone of the complexes $\mathcal{S}_{\bullet}$ and $\mathcal{R}_{\bullet}$, then $\mathcal{U}_{\bullet}$ contains $\mathcal{S}_{\bullet}$. This is the situation of Definition \ref{generalgaeta} and Theorems A \& B.

\medskip\noindent
By the \textit{dual} of $\mathcal{R}_\bullet$, we mean the complex 
$$\mathcal{R}_\bullet^*: \quad  \bigoplus_{i=l+1}^{m+k} \mathcal{O}_{\PP^2}(d_{1,\ell_i})\overset{A^T}{\xrightarrow{\hspace{7mm}}} \bigoplus_{i=n+1}^{m} \mathcal{O}_{\PP^2}(d_{2,j_i}).$$
We can similarly define the \textit{dual} for any complex of line bundles.

\medskip\noindent
\subsubsection{Gaeta resolution.} \label{GaetaRes} 
If a sheaf $U$ is general in its moduli space $M(\xi)$, then its minimal free resolution is often called \textit{Gaeta resolution} (after the Spanish mathematician Federico Gaeta \cite{Ciliberto2010}) and its Betti numbers depend only on $\xi$.  Indeed, if $\chi\left(\mathcal{T}_{\mathbb{P}^2}(-d-1), U\right) \geq 0$, where $d$ is the minimal integer such that $h^0(\mathbb{P}^2,U(d))>0$ and $\mathcal{T}_{\mathbb{P}^2}$ is the tangent sheaf of $\PP^2$, then the Gaeta resolution of $U$ is 
\[0 \xrightarrow{\hspace{7mm}} \mathcal{O}_{\PP^2}(-d-2)^{-n_3}  \oplus \mathcal{O}_{\PP^2}(-d-1)^{n_2} \overset{M}{\xrightarrow{\hspace{7mm}}}   \mathcal{O}_{\PP^2}(-d)^{n_1} \xrightarrow{\hspace{7mm}} U \xrightarrow{\hspace{7mm}} 0,\] where 
$n_1 = \chi\left(\mathcal{O}_{\PP^2}(-d),U\right)$, $n_2=\chi\left(\mathcal{T}_{\mathbb{P}^2}(-d-1), U\right)$, $n_3 = \chi\left(\mathcal{O}_{\PP^2}(-d+1),U\right)$, and $M$ is general.
Similarly, if $\chi\left(\mathcal{T}_{\mathbb{P}^2}(-d-1), U\right) \leq 0$, then the Gaeta resolution of $U$ is \[0 \xrightarrow{\hspace{7mm}} \mathcal{O}_{\PP^2}(-d-2)^{-n_3}  \overset{N}{\xrightarrow{\hspace{7mm}}} \mathcal{O}_{\PP^2}(-d-1)^{-n_2} \oplus   \mathcal{O}_{\PP^2}(-d)^{n_1} \xrightarrow{\hspace{7mm}} U \xrightarrow{\hspace{7mm}} 0,\] with $N$ general.
A Gaeta resolution is called \textit{pure} if it contains only two distinct line bundles.

\medskip\noindent
If $U$ is a locally-free sheaf, the dual $U^*$ has a minimal free resolution, which by dualizing again induces the following exact sequence 
\[0 \xrightarrow{\hspace{7mm}} U\xrightarrow{\hspace{7mm}}   \bigoplus_{i=1}^{m+k} \mathcal{O}_{\PP^2}(d_{1,i})\xrightarrow{\hspace{7mm}}\bigoplus_{i=1}^{k} \mathcal{O}_{\PP^2}(d_{2,i}) \xrightarrow{\hspace{7mm}} 0.\]

\medskip\noindent
\subsubsection{Generalized Gaeta resolution.}
\label{subsubsec: GaetaTriangle}
The general sheaf $U$ in the moduli space $M(\xi)$ admits a resolution in terms of exceptional bundles 
\cite[Proposition 5.3]{CHW}. 
Indeed, if $E_{\alpha\cdot\beta}$ is the controlling exceptional bundle of $\xi$ and $\chi(E_{-(\alpha\cdot\beta)}, U) \geq 0$, then $U$ has the generalized Gaeta resolution: 
\begin{equation}\label{res1}
0 \xrightarrow{\hspace{7mm}} E_{-\alpha-3}^{m_3}   \xrightarrow{\hspace{7mm}} E_{-\beta}^{m_2}  \oplus E_{-(\alpha\cdot\beta)}^{m_1} \xrightarrow{\hspace{7mm}} U \xrightarrow{\hspace{7mm}} 0,
\end{equation}
where $m_1 = \chi\left(E_{-(\alpha\cdot\beta)},U\right)$, $m_2 = -\chi\left(E_{-(\alpha\cdot(\alpha\cdot\beta))},U\right)$, and $m_3= -\chi\left(E_{-\alpha},U\right)$.

\medskip\noindent
Similarly,  if $E_{\alpha\cdot\beta}$ is the controlling exceptional bundle of $\xi$ and $\chi\left(E_{-(\alpha\cdot\beta)}, U\right) \leq 0$ for a general sheaf $U \in M(\xi)$, then $U$ has the generalized Gaeta resolution:
\begin{equation}\label{res2}
0 \xrightarrow{\hspace{7mm}} E_{-(\alpha\cdot\beta)-3}^{m_1} \oplus E_{-\alpha-3}^{m_3}   \xrightarrow{\hspace{7mm}} E_{-\beta}^{m_2}  \xrightarrow{\hspace{7mm}} U \xrightarrow{\hspace{7mm}} 0, 
\end{equation}
where $m_1=-\chi\left(E_{-(\alpha\cdot\beta)},U\right)$,  $m_2 = \chi\left(E_{-\beta},U\right)$,  and $m_3 = \chi\left(E_{-((\alpha\cdot\beta)\cdot \beta)},U\right)$.

\medskip\noindent
We have used relative Euler characteristics to express the exponents above as they occur in the Beilinson spectral sequence. 
Notice that if $E_{\alpha\cdot \beta}$ is a line bundle, the resolutions above coincide with the Gaeta resolution.

\subsection{Birational geometry}\label{BG}
Let us briefly recall the notions we will study in birational geometry and refer the reader to \cite{Lazarsfeld} for a thorough treatment.

\medskip\noindent
The N\'eron-Severi space $\mathrm{NS}(X)_\mathbb{Q}$ of a smooth algebraic variety $X$ is the vector space over $\mathbb{Q}$ of divisors modulo numerical equivalence. Inside the N\'eron-Severi space, the closure of the convex cone spanned by the classes of subvarieties is called the \textit{effective cone}, denoted $\mathrm{Eff}(X)$.
This cone decomposes according to the positivity of its divisor classes; specifically, according to their stable base locus.
The \textit{stable base locus} of a divisor class $D\in \EFF(X)$, denoted $\mathbf{B}(D)$, is the intersection of all divisors linearly equivalent to any positive integer multiple of $D$. That is \[\mathbf{B}(D):= \bigcap_{m>0}\bigcap_{D' \in \vert mD \vert}D'.\]
The stable base locus of an effective divisor has codimension at least 1. The \textit{movable cone}, $\Mov(X)$, is defined as the closure of the locus of classes in $\EFF(X)$ whose stable base locus has codimension $\geq 2$. 

\medskip\noindent
Moduli spaces of sheaves on $\mathbb{P}^2$ are Mori dream spaces \cite[Theorem 2.4]{CHW}. This implies that there is an open subset of the N\'eron-Severi space where the stable base locus is locally constant whose complement is contained in a union of finitely many hyperplanes \cite{ELMNP}. A connected component of this open subset is called a \textit{chamber}. The union of such chambers of $\EFF(X)$ is called the \textit{stable base locus decomposition}, which we abbreviate SBLD.

\subsection{Bridgeland stability}
If $\xi$ is a stable character, then the Picard rank of the moduli space satisfies $\rho(M(\xi))\le 2$. If $\rho(M(\xi))=2$, then the stable base locus decomposition of $M(\xi)$ has been worked out using Bridgeland stability \cite{LZ}. In that same paper, it is shown that each sheaf $U\in M(\xi)$ comes with one (or more) \textit{destabilizing object}, which is a torsion-free sheaf $F$ admitting a map $F \xrightarrow{\hspace{3mm}} U$. 
This leads to a distinguished triangle in the derived category $D^b(\mathbb{P}^2)$ (which explains our notation $(F,W)$-cone.) $$F\xrightarrow{\hspace{3mm}} U \xrightarrow{\hspace{3mm}} W \xrightarrow{\hspace{3mm}} \cdot$$

\medskip\noindent
Indeed, following \cite{ABCH}, if the pair $(\mu_{s,t},\mathcal{A}_s)$ is a stability condition on $M(\xi)$, the triangle above yields an inclusion of objects in the category $\mathcal{A}_s$ of the same $(s,t)$-slope, and $U$ is $(s,t)$-semi-stable, then $F$ is a destabilizing object, $W$ is a destabilizing quotient and the triangle a \textit{destabilizing sequence}. In the space of stability conditions, the $(s,t)$-plane, a wall $$W_{F,E}=\{(s,t)\ | \ E \ \mbox{is (s,t)-semi-stable and destabilized by } F\}$$ is a semi-circle \cite[Section 6]{ABCH}. Thus, it has a center $x$ and can be denoted by $W_x$.

\medskip\noindent
Complexes in the category $\mathcal{A}_s$ have cohomology only in degrees $-1$ and $0$. As a result,  we often denote a complex which sits in degrees $0$ and $1$ by $W[-1]$; that is, as the $-1$ shift of an object $W$ sitting in degrees $-1$ and $0$. For example, the vector bundle resolution in \eqref{res1} can be written in $D^b(\PP^2)$ as an exact triangle (called the \textit{Gaeta triangle}): 
\begin{equation*}
W[-1]\to E_{-(\alpha\cdot\beta)}^{m_1}\to U\to \cdot,
\end{equation*}
where $W$ is an $(s,t)$-semi-stable complex sitting in degrees $-1$ and $0$.

\medskip
\subsection{Categories}
In this paper, we work in three categories:
\begin{enumerate}
    \item $\mathrm{Coh}(\mathbb{P}^2)$ - the category of coherent sheaves on the plane,
    \item $\mathrm{Kom}(\mathbb{P}^2)$ - the category of complexes of coherent sheaves on the plane, and
    \item $D^b(\mathbb{P}^2)$ - the bounded derived category of complexes of coherent sheaves on the plane.
\end{enumerate}
\medskip\noindent
An object in $\mathrm{Coh}(\mathbb{P}^2)$ can be considered as an object in $\mathrm{Kom}(\mathbb{P}^2)$ or $D^b(\mathbb{P}^2)$, thinking of it as a complex sitting in degree 0.
Thus, for any two coherent sheaves $W$ and $F$, we have that 
\[\mathrm{Hom}_{\mathrm{Coh}(\mathbb{P}^2)}(W,F) = \mathrm{Hom}_{\mathrm{Kom}(\mathbb{P}^2)}(W,F) = \mathrm{Hom}_{D^b(\mathbb{P}^2)}(W,F).\]

\medskip\noindent
A bounded complex in $\mathrm{Kom}(\mathbb{P}^2)$ is an object in $D^b(\mathbb{P}^2)$, but in general $\mathrm{Hom}(W, F)_{\mathrm{Kom}(\PP^2)}\ne \mathrm{Hom}_{D^b(\PP^2)}(W, F)$. In certain circumstances, we can relate them as we show next. This lemma will be used in proving Theorem \ref{MAINdetailed}.

\begin{lem}\label{derivedHom} 
    Consider 2-term complexes $\;\;W_\bullet\;:\;\;W_{0}\overset{A}{\xrightarrow{\hspace{5mm}}} W_1 \;\;\;\text{ and }\;\;\; F_\bullet\;:\;\;F_{-1}\overset{B}{\xrightarrow{\hspace{5mm}}} F_0$.
    Let \[K= \mathrm{Hom}_{\mathrm{Coh}(\mathbb{P}^2)}(W_0,F_{-1})\oplus \mathrm{Hom}_{\mathrm{Coh}(\mathbb{P}^2)}(W_1,F_{0}).\]
    If $\mathrm{ext}^i_{\mathrm{Coh}(\mathbb{P}^2)}(W_k,F_\ell) =0$ (for all $i>0$ for all $k$ and $\ell$) and $\mathrm{hom}_{\mathrm{Coh}(\mathbb{P}^2)}(W_1,F_{-1}) =0$, then 
    \[\mathrm{Hom}_{D^b(\mathbb{P}^2)}(W,F) = \mathrm{Hom}_{\mathrm{Kom}(\mathbb{P}^2)}(W,F)\slash K\] where we consider $K$ as a subgroup of $\mathrm{Hom}_{\mathrm{Kom}(\mathbb{P}^2)}(W,F)$ via the map $(h_0,h_1)\mapsto h_1 \circ A-B\circ h_0$.
\end{lem}

\begin{proof}
This proof follows \cite[Lemma 5.4]{CHW} closely after shifting the indices of $W$ and noting that $\mathrm{Hom}_{\mathrm{Kom}(\mathbb{P}^2)}(W, F) = \mathrm{Hom}_{\mathrm{Coh}(\mathbb{P}^2)}(W_0,F_0)$ by degree considerations.
\end{proof}

\medskip\noindent
This lemma says that every homomorphism in the derived category between the complexes $W_\bullet$ and $F_\bullet$ comes from a homomorphism of complexes, $i.e.$, given a derived homomorphism, we can pick a representative homomorphism of complexes which descends to it. When every term is free, this representative is unique up to row and column reduction as modding out by $K$ is modding out by a subset of such row and column reductions.

\medskip
\section{Proof of Theorem \ref{MAINdetailed} }
\label{sec: eff}

\noindent
In this section, we prove our main result Theorem \ref{MAINdetailed}. 

\medskip\noindent
Let $\xi$ be a fixed log Chern character such that the moduli space $M(\xi)$ is positive dimensional with Picard rank 2. 
 The proof of Theorem \ref{MAINdetailed} depends on the sign of the pairing $ \chi(E_{-(\alpha\cdot\beta)}, U)$, where $E_{\alpha\cdot\beta}$ denotes the controlling exceptional of $\xi$ and $U\in M(\xi)$. We address each case separately.
We will prove Part (c) first, and use it as a template for showing Parts (a) and (b).


\medskip
\subsection{Exceptional case: $\chi(E_{-(\alpha\cdot\beta)},U)=0$.}
We will show the following: a sheaf $U$ admits a resolution as in \eqref{exceptionalgaetatriangle} 
if and only if the minimal free resolution $U_{\bullet}$ is the mapping cone of the minimal free resolution of $E_{-\beta}^{m_2}$ and the dual of the minimal free resolution of $E_{\alpha + 3}^{m_3}$.

\begin{proof}[Proof of Theorem \ref{MAINdetailed}, Part (c).]

Assume that $U$ has a generalized Gaeta resolution \eqref{exceptionalgaetatriangle}
\begin{equation}\label{res4}
0 \xrightarrow{\hspace{7mm}} (E_{-3-\alpha})^{m_3}  \xrightarrow{\hspace{7mm}} \left(E_{-\beta}\right)^{m_2}  \xrightarrow{\hspace{7mm}} U \xrightarrow{\hspace{7mm}} 0
\end{equation}
where $m_3= -\chi\left( E_{-\alpha},U\right) = \chi\left( E_{-((\alpha.\beta).\beta)},U\right)$ and $m_2 = \chi(E_{-\beta},U) = -\chi(E_{-(\alpha.(\alpha.\beta))},U)$. 

If $E_{\alpha \cdot \beta}$ is a twist of the tangent bundle $\mathcal{T}_{\PP^2}(k)$ or a line bundle, then $E_{\alpha}$ and $E_{\beta}$ are line bundles and the result follows as Resolution \eqref{res4} is the Gaeta resolution. Let us assume below that $\textrm{rk}(E_{\alpha\cdot \beta})>2$, and write Resolution \eqref{res4} as follows:
\[0 \xrightarrow{\hspace{7mm}} W \overset{f}{\xrightarrow{\hspace{7mm}}} F \xrightarrow{\hspace{7mm}} U\xrightarrow{\hspace{7mm}} 0, \] 
where $F= E_{-\beta}^{m_2}$ and $W=E_{-\alpha-3}^{m_3}$. 

We aim to show that \eqref{res4} gives rise to the minimal free resolution of $U$ using the following strategy: 
In $D^b(\PP^2)$, we replace $W$ and $F$ by equivalent complexes of line bundles which come from minimal free resolutions. Then, we show that $f$ lifts to a map of such complexes and its mapping cone is a complex involving only line bundles. Finally, it suffices to argue that this mapping cone is, in fact, the minimal free resolution of $U$.

To follow that strategy, recall that the sheaf $F$ has a minimal free resolution, which as a complex in degrees $-1,0$ is \[F_{\bullet}: \quad F_{-1} \overset{B}{\xrightarrow{\hspace{7mm}}} F_{0},\] with cohomology $\mathcal{H}^{0}(F_\bullet)=coker(B)=F$ and zero in all other degrees.
Notice that $F$, as a complex in degree $0$ and zero elsewhere, and $F_\bullet$ are equivalent complexes.
Similarly, by dualizing the minimal free resolution of $W^*$, we have the complex
\begin{equation}\label{W}
W_{\bullet}: \quad  W_{0} \overset{A}{\xrightarrow{\hspace{7mm}}} W_{1} 
\end{equation}
with cohomology $\mathcal{H}^{0}(W_\bullet)=ker(A)=W$ and zero in all other degrees. 
Again, $W_\bullet$ and $W$ are equivalent in $D^b(\PP^2)$. 

As outlined above, we want to lift the map $f$ to the category of complexes using Lemma \ref{derivedHom}, so we need to show that $\mathrm{ext}_{\mathrm{Coh}(\PP^2)}^i(W_k,F_\ell)=0$ for all $i>0$. 
In order to do this, consider the integer $d$ such that $E_{\alpha \cdot \beta} = E^*(d)$ and $\mu_E \in [0,1)$.
Since $\mathrm{rk}(E_{\alpha\cdot\beta})>2$, $-\alpha$ and $-\beta$ lie in the same half-integer interval, either $[-d,-d+\frac{1}{2}]$ or $[-d+\frac{1}{2},-d+1]$. 
This implies that the Gaeta resolution of $E_{-\beta}$ can be written in terms of $\mathcal{O}_{\PP^2}(-d-2)$, $\mathcal{O}_{\PP^2}(-d-1),\mathcal{O}_{\PP^2}(-d)$, and that of $E_{\alpha+3}$ in terms of $\mathcal{O}_{\PP^2}(d)$, $\mathcal{O}_{\PP^2}(d+1),\mathcal{O}_{\PP^2}(d+2)$. 

Now, the Beilinson spectral sequence and the previous paragraph tell us that if $\mu_E \in [0,\frac{1}{2})$, 

\begin{equation*}
\begin{aligned}
    F_{-1} &= \mathcal{O}_{\PP^2}(-d-2)^{-\chi(\mathcal{O}_{\PP^2}(-d+1),F)} \oplus \mathcal{O}_{\PP^2}(-d-1)^{\chi(E_{\frac{1}{2}}(-d) ,F)},\\
    F_{0} &= \mathcal{O}_{\PP^2}(-d)^{\chi(\mathcal{O}_{\PP^2}(-d),F)},\\
    W_{0} &= \mathcal{O}_{\PP^2}(-d-2)^{\chi(\mathcal{O}_{\PP^2}(-d+1),W)} \oplus \mathcal{O}_{\PP^2}(-d-1)^{-\chi(E_{\frac{1}{2}}(-d),W)} ,\text{ and}\\
    W_{1} &= \mathcal{O}_{\PP^2}(-d)^{-\chi(\mathcal{O}_{\PP^2}(-d),W)},\\
\end{aligned}
\end{equation*}
and if $\mu_E \in \left[\frac{1}{2},1\right)$, 
\begin{equation*}
\begin{aligned}
    F_{-1} &= \mathcal{O}_{\PP^2}(-d-2)^{-\chi(\mathcal{O}_{\PP^2}(-d+1),F)},\\
    F_{0} &= \mathcal{O}_{\PP^2}(-d-1)^{-\chi(E_{\frac{1}{2}}(-d),F)}\oplus \mathcal{O}_{\PP^2}(-d)^{\chi(\mathcal{O}_{\PP^2}(-d),F)},\\
    W_{0} &=  \mathcal{O}_{\PP^2}(-d-2)^{\chi(\mathcal{O}_{\PP^2}(-d+1),W)},\text{ and }\\
    W_1 &= \mathcal{O}_{\PP^2}(-d-1)^{\chi(E_{\frac{-1}{2}}(-d),W)} \oplus\mathcal{O}_{\PP^2}(-d)^{-\chi(\mathcal{O}_{\PP^2}(-d),W)}.\\
\end{aligned}
\end{equation*}

Since $W_k$ and $F_\ell$ are direct sums of line bundles with 3 consecutive degrees,\\ $\mathrm{ext}_{\mathrm{Coh}(\PP^2)}^i(W_k,F_\ell) = 0$ for all $k,\ell$ and $i>0$.  
Similarly, $\mathrm{hom}_{\mathrm{Coh}(\PP^2)}(W_1, F_{-1})=0$ by degree considerations so $f$ induces a map of complexes $f_\bullet:W_{\bullet} \xrightarrow{\hspace{3mm}} F_{\bullet}$ by Lemma \ref{derivedHom}.

In the category of complexes, $f_\bullet$ maps $W_i \xrightarrow{\hspace{3mm}} F_i$ in the following way:
\begin{center}
    \begin{tikzcd}
        W_{\bullet}:  0 \arrow{rd}{0} \arrow[r] & W_{0} \arrow{r}{A} \arrow{rd}{\tilde{f}} & W_1 \arrow{rd}{0} \arrow{r}{d_1}  &  0 \\  
       F_{\bullet}:   0 \arrow[r] & F_{-1} \arrow{r}{B}  & F_{0} \arrow{r}{d_{0}}  &  0. 
    \end{tikzcd}
\end{center} 

Consequently, we can think of $f$ as a map of complexes defined by $\Tilde{f}$ in degree $0$ and zero elsewhere. Note that $W_{0}$ and $F_{0}$ are sums of line bundles, so the map $\Tilde{f}$ can be represented by a matrix, denoted by $C$.
Observe that ${f_\bullet}$ and $f$ represent the same map in $D^b(\PP^2)$. This means that they have equivalent mapping cones. It follows that $U$, as the cohomology of the mapping cone of $f$ in degree $0$, is also the cohomology of the mapping cone of $f_\bullet$ in the same degree. Let us write this latter cone explicitly as it will be the minimal free resolution of $U$. 

Recall that the minimal free resolution of $F$ yields a matrix $B$, and similarly we get a matrix $A^T$ from the minimal free resolution of $W^*$. The matrix $C$ is induced by $\Tilde{f}$. The mapping cone of $f_\bullet$ is therefore the complex
\begin{equation}\label{eq: summand0}
Z_{\bullet}: \quad W_{0}\oplus F_{-1} \overset{M}{\xrightarrow{\hspace{7mm}}} W_1\oplus F_{0}
\end{equation} 
sitting in degrees $-1,0$, and where the matrix $M$ can be written as follows
\[M=\left(
\begin{array}{c|c}
0  & A \\ \hline
 B & C 
\end{array}\right).\]

By construction, $\mathcal{H}^{0}(Z_\bullet)\cong  U$ and $\mathcal{H}^{i}(Z_\bullet)\cong  0$ for $i\neq 0$, so we get the following short exact sequence of free sheaves which is a free resolution of $U$:
$$0\xrightarrow{\hspace{7mm}} W_{0}\oplus F_{-1} \overset{M}{\xrightarrow{\hspace{7mm}}} W_1\oplus F_{0}\xrightarrow{\hspace{7mm}} U \xrightarrow{\hspace{7mm}} 0.$$

It now suffices to notice that this free resolution of $U$ is minimal as $W_{0}\oplus F_{-1}$ and $F_{0}\oplus W_1$ do not share a line bundle of the same degree by the computation above.

This construction exhibits the minimal free resolution of $U$ as an $(F,W)$-cone, which completes this part of the proof.
Note that the minimal free resolution of $U$ above is the Gaeta resolution as the Euler characteristic is additive on the short exact sequence \eqref{res4} and this is the only minimal free resolution where only 3 consecutive degrees appear among the line bundles, and no degree appears both in the generators and the relations. 

\medskip
Now, let us argue the converse. Assume that the minimal free resolution of $U$ is the mapping cone $U_{\bullet}=MC(f)_{\bullet}$, in the category of complexes, between complexes formed with sums of line bundles $$W_{\bullet}\overset{f}{\to}F_{\bullet}\to U_{\bullet}\to\cdot$$ such that the only non-zero cohomology is $\mathcal{H}^{0}(W_{\bullet})=E_{-\alpha-3}^{m_3}$ and $\mathcal{H}^0(F)=E_{-\beta}^{m_2}$. Then
the induced sequence in cohomology
$$0\to \mathcal{H}^{0}(W)\to \mathcal{H}^{0}(F)\to \mathcal{H}^{0}(U)\to 0$$
 yields Resolution (3) as desired.
This completes the proof of Part (c) of Theorem \ref{MAINdetailed}.
\end{proof}

\medskip\noindent
Notice that either the matrix $B$ or $A$ in the previous proof may be a matrix with size zero. This occurs when either $E_{-\beta}$ or $E_{-\alpha-3}$ is a line bundle. In case both $E_{-\beta}$ and $E_{-\alpha-3}$ are line bundles, Resolution \eqref{res4} coincides with the Gaeta resolution.

\medskip\noindent
The following example exhibits how an ideal sheaf of $n$ points $\mathcal{I}_Z$ can be $E$-admissible, for an exceptional bundle $E$ of high rank. This reveals how the subcomplexes of $(\mathcal{I}_Z)_{\bullet}$ can be very complicated as the value of $n$ increases.

\begin{example}\label{ex::reallyBigExample}
Let $\xi=(1,0,2896)$. Then $M(\xi)\cong \PP^{2[2896]}$ is the Hilbert scheme of $n=2896$ points on the plane. The Chern character of the controlling exceptional bundle is $e_{\alpha\cdot\beta}=(194,\frac{14475}{194},\frac{37635}{75272})$ and, by Riemann-Roch, we have that $(\xi, e_{\alpha\cdot \beta})=0$.  Thus, it follows that $\alpha=\tfrac{373}{5}$ and $\beta=\tfrac{970}{13}$.

\medskip\noindent
Consider $U\in M(\xi)$ general. By Theorem \ref{MAINdetailed}, Part (c) the minimal free resolution of $U$
$$0\xrightarrow{\hspace{7mm}} \mathcal{O}_{\PP^2}(-77)^{46} \overset{M}{\xrightarrow{\hspace{7mm}}} \mathcal{O}_{\PP^2}(-76)^{17} \oplus \mathcal{O}_{\PP^2}(-75)^{30} \xrightarrow{\hspace{7mm}} U \xrightarrow{\hspace{7mm}} 0,$$ 
is the mapping cone of the following two complexes:
$$E_{-\beta}^2: \quad \mathcal{O}_{\PP^2}(-77)^{6} \xrightarrow{\hspace{7mm}} \mathcal{O}_{\PP^2}(-76)^{2}\oplus  \mathcal{O}_{\PP^2}(-75)^{30}  $$ and 
$$E_{-\alpha-3}^5: \quad \mathcal{O}_{\PP^2}(-77)^{40}\xrightarrow{\hspace{7mm}} \mathcal{O}_{\PP^2}(-76)^{15}.$$ 
This implies that the matrix $M$ has the following shape:
\[M=\left(
\begin{array}{c|c}
0  & A \\ \hline
 B & C 
\end{array}\right),\]
and $B$ can be thought of as the matrix in the minimal free resolution of $E_{-\beta}^2$.

\medskip\noindent
Corollary \ref{cor: EffRigid} will imply that the $(E_{-\beta}^2,E_{-\alpha-3}^5)$-cone structure of $(\mathcal{I}_Z)_{\bullet}$ characterizes the stable base locus of the first SBLD chamber of $\PP^{2[2896]}$. 
Indeed, given a matrix similar to $M$, but in which $B$ fails to be full-rank, we can construct a sheaf $U'$ contained in the base locus of the first primary chamber of $\EFF(M(\xi))$. Here, $U'$ has the same Betti numbers as $U$, which exemplifies that the Betti numbers fail to characterize the stable base locus.
\end{example}

\subsection{The positive case: $\chi(E_{-(\alpha\cdot \beta)}, U)>0$.}\label{Positivecase}
In this subsection, we prove Part (a) of Theorem \ref{MAINdetailed}. 

\medskip\noindent
The argument of this proof is an iteration of an argument in the previous subsection. 
An additional difficulty arises here because we cannot treat $W$ as a sheaf as in the previous subsection, because it is a 2-term complex of vector bundles.

\begin{proof}[Proof of Theorem \ref{MAINdetailed}, Part (a)] 
Part of this theorem aims to conclude that Bridgeland semi-stable objects form the minimal free resolution via mapping cones. Then, in the following proof we keep track of the grading to clarify this, by writing $W_\bullet[-1]$ instead of $W_\bullet$, etc.

Let $U$ be a sheaf in $M(\xi)$, and set 
$$m_1=\chi(E_{-(\alpha\cdot\beta)},U),\quad m_2=-\chi(E_{-(\alpha\cdot(\alpha\cdot\beta))},U), \quad \mbox{and} \quad m_3=-\chi(E_{-\alpha},U).$$

We prove the sufficient direction first. Assume $U$ admits a resolution as in \eqref{res1}.  The Gaeta triangle induced by Resolution \eqref{res1} is 
\[ W_\bullet[-1] \overset{f_\bullet}{\xrightarrow{\hspace{7mm}}} F\xrightarrow{\hspace{7mm}} U\xrightarrow{\hspace{7mm}} \cdot, \] where 
$F=E_{-(\alpha \cdot \beta)}^{m_1}$ and $W_\bullet$ is the 2-term complex in the derived category $D^b(\PP^2)$, $$W_{\bullet}: \quad E_{-3-\alpha}^{m_3} \overset{t}{\xrightarrow{\hspace{7mm}}} E_{-\beta}^{m_2}$$ sitting in degrees $-1$ and $0$.  The sheaf $F$ has a minimal free resolution, which can be written as the complex in degrees $-1,0$ \[F_{\bullet}: \quad  F_{-1} \overset{B}{\xrightarrow{\hspace{7mm}}} F_{0} ,\] with cohomology $\mathcal{H}^{0}(F_\bullet)=coker(B)=F$ and zero in all other degrees. This complex is equivalent to $F$ in $D^b(\PP^2)$. Let us now address $W_{\bullet}$.

Thinking of $G[-1]=E_{-3-\alpha}^{m_3}$ and $H=E_{-\beta}^{m_2}$ as vector bundles, let us consider the following minimal free resolutions 
\[0 \xrightarrow{\hspace{7mm}} G[-1]_{1}^* \xrightarrow{\hspace{7mm}} G[-1]_{0}^* \xrightarrow{\hspace{7mm}} G[-1]^*\xrightarrow{\hspace{7mm}} 0 \;\;\text{ and}\]  
\[0 \xrightarrow{\hspace{7mm}} H_{-1} \xrightarrow{\hspace{7mm}} H_{0} \xrightarrow{\hspace{7mm}} H \xrightarrow{\hspace{7mm}} 0,\]
where $G[-1]^*$ stands for the dual vector bundle of $G[-1]$. These are short exact sequences of sheaves 
and the labeling (shift and subscript) allows us to consider them as complexes in $D^b(\PP^2)$ with the correct grading: 
\[H_\bullet: \quad  H_{-1} \xrightarrow{\hspace{7mm}} H_0  \;\;\text{ and}\]\[G[-1]_\bullet: \quad  G[-1]_0 \xrightarrow{\hspace{7mm}} G[-1]_1.\] These complexes are equivalent in $D^b(\PP^2)$ to the sheaves $H$ and $G[-1]$ in degree $0$.

In order to lift the map $t$, using Lemma \ref{derivedHom}, to a map in the category of complexes
\[t_{\bullet}:G[-1]_{\bullet}\xrightarrow{\hspace{5mm}} H_{\bullet}\] 
we must compute the line bundles that appear in them. 
To this end, consider the integer $d$ such that $E_{\alpha \cdot \beta} = E^*(d)$ and $\mu_E \in [0,1)$. 
If $\mathrm{rk}(E_{\alpha\cdot\beta})=1$, then the result holds as \eqref{gaetatrianglepositive} is the Gaeta resolution.
If $\mathrm{rk}(E_{\alpha\cdot\beta})=2$, i.e., if $E_{\alpha\cdot\beta}= E_{\frac{1}{2}}(d-1)$ then $H_{-1}=0$, $G[-1]_1=0$, $H_0 = \mathcal{O}_{\PP^2}(-d)^{m_2}$, and $G[-1]_0 = \mathcal{O}_{\PP^2}(-d-2)^{m_3}$, which is a trivial subcase of what follows when $\mathrm{rk}(E_{\alpha\cdot\beta})>2$.
When $\mathrm{rk}(E_{\alpha\cdot\beta})>2$, similarly to Part (c),  
we have that 
if $\mu_E \in [0,\frac{1}{2})$, then

\begin{equation*}
\begin{aligned}
    H_{-1} &= \mathcal{O}_{\PP^2}(-d-2)^{-\chi(\mathcal{O}_{\PP^2}(-d+1),H)} \oplus \mathcal{O}_{\PP^2}(-d-1)^{\chi(E_{\frac{1}{2}}(-d) ,H)},\\
    H_{0} &= \mathcal{O}_{\PP^2}(-d)^{\chi(\mathcal{O}_{\PP^2}(-d),H)},\\
    G[-1]_{0} &= \mathcal{O}_{\PP^2}(-d-2)^{\chi(\mathcal{O}_{\PP^2}(-d+1),G)} \oplus \mathcal{O}_{\PP^2}(-d-1)^{-\chi(E_{\frac{1}{2}}(-d),G)} ,\\
    G[-1]_1 &= \mathcal{O}_{\PP^2}(-d)^{-\chi(\mathcal{O}_{\PP^2}(-d),G)},\\
\end{aligned}
\end{equation*}
and if $\mu_E \in \left[\frac{1}{2},1\right)$, then
\begin{equation*}
\begin{aligned}
    H_{-1} &= \mathcal{O}_{\PP^2}(-d-2)^{-\chi(\mathcal{O}_{\PP^2}(-d+1),H)},\\
    H_{0} &= \mathcal{O}_{\PP^2}(-d-1)^{-\chi(E_{\frac{1}{2}}(-d) ,H)}\oplus \mathcal{O}_{\PP^2}(-d)^{\chi(\mathcal{O}_{\PP^2}(-d),H)},\\
    G[-1]_{0} &= \mathcal{O}_{\PP^2}(-d-2)^{\chi(\mathcal{O}_{\PP^2}(-d+1),G)} ,\\
    G[-1]_1 &= \mathcal{O}_{\PP^2}(-d-1)^{\chi(E_{\frac{1}{2}}(-d),G)}\oplus \mathcal{O}_{\PP^2}(-d)^{-\chi(\mathcal{O}_{\PP^2}(-d),G)}.\\
\end{aligned}
\end{equation*}

This computation allows us to apply Lemma \ref{derivedHom} in order to lift the map $t$ to the map of complexes $t_{\bullet}$.
In the category of complexes, $t_\bullet$ maps $G[-1]_i \xrightarrow{\hspace{3mm}} H_i$ as follows
\begin{center}
    \begin{tikzcd}
        G[-1]_{\bullet}:   &0  \arrow{rd}{0} \arrow[r] & G[-1]_{0} \arrow{r}{} \arrow{rd}{\Tilde{t}} & G[-1]_1 \arrow{rd}{0} \arrow{r}{}  &  0 \\  
       H_{\bullet}:      &0 \arrow[r] & H_{-1} \arrow{r}{}  & H_{0} \arrow{r}{}  &  0. 
    \end{tikzcd}
\end{center} 

Consequently, we can think of the map in $W$ as a map of complexes defined by $\Tilde{t}$ in degree $0$ and zero elsewhere. Note that $G[-1]_{0}$ and $H_{0}$ are sums of line bundles, so $\Tilde{t}$ can be represented by a matrix.

Following the construction of \eqref{eq: summand0}, $W_\bullet$ is equivalent in $D^b(\PP^2)$ to the following complex 
$$W_{\bullet}: \quad H_{-1}\oplus G[-1]_{0}\overset{A}{\xrightarrow{\hspace{7mm}}} H_{0}\oplus G[-1]_{1},$$ which sits in degrees $-1,0$ and whose factors are sums of line bundles.

Observe that $W[-1]_\bullet$ is now in the same form as in \eqref{W}. Also, note that $W[-1]_\bullet$ contains the minimal free resolution of $E_{-\beta}^{m_2}$. Proceeding identically as in the proof of Theorem \ref{MAINdetailed} Part (c) using $W[-1]_\bullet$ in place of $W_\bullet$, we conclude that the mapping cone of $f_{\bullet}$ yields the minimal free resolution of $U$ as $F$-admissible and the residual complex $W[-1]$ is a $(E_{-\beta}^{m_2}, E_{-\alpha - 3}^{m_3})$-cone.

\medskip
Now, let us argue the converse. Assume that the minimal free resolution of $U$ is the mapping cone $U_{\bullet}=MC(g)_{\bullet}$, in the category of complexes, of complexes formed with sums of line bundles $$W_{\bullet}[-1]\overset{g}{\to} F_{\bullet}\to U_{\bullet}$$ such that $F_{\bullet}$ is the minimal free resolution of $F=E_{-(\alpha.\beta)}^{m_1}$ and $W_{\bullet}$ is a  $(E_{-\beta}^{m_2},E_{-\alpha - 3}^{m_3})$-cone. 
In other words, in the derived category, $W_{\bullet}[-1]$ is equivalent to the complex $E_{-\alpha-3}^{m_3} \to E_{-\beta}^{m_2}$ in degrees $0,1$, and $F_\bullet$ is equivalent to the sheaf $E_{-(\alpha.\beta)}^{m_1}$ in degree $0$. 
By Lemma \ref{derivedHom}, $g$ lifts to a map between these complexes.
Since that map of complexes has a mapping cone whose only cohomology is $U$ in degree $0$ and the correct bundles are involved, this induces Resolution \eqref{res1}, which completes the proof of Part (a) of Theorem \ref{MAINdetailed}.

\end{proof}

\begin{example}\label{ex::notSoBigExample}
Let $\xi=(1,0,165)$. Then $M(\xi)\cong \PP^{2[165]}$ is the Hilbert scheme of $n=165$ points on the plane. In this case, the controlling exceptional slope is $\alpha\cdot \beta =\tfrac{83}{5}$ and Riemman-Roch implies that the pairing $(\xi, e_{\alpha\cdot \beta})>0$. Hence, we can determine the syzygies' behavior of a general $Z\in\PP^{2[165]}$ using Theorem \ref{MAINdetailed}, Part (a).

\medskip\noindent
If $U\in M(\xi)$ is general, then by Theorem \ref{MAINdetailed} the Gaeta resolution of $U$ is $E_{-(\alpha\cdot \beta)}$-admissible and the residual complex has $\OO_{\PP^2}(-17)^2$ as a factor. That is, it can be written as 
$$0\xrightarrow{\hspace{7mm}} \mathcal{O}_{\PP^2}(-19)^{10} \overset{M}{\xrightarrow{\hspace{7mm}}} \mathcal{O}_{\PP^2}(-18)^{3} \oplus \mathcal{O}_{\PP^2}(-17)^{8}\xrightarrow{\hspace{7mm}} U \xrightarrow{\hspace{7mm}} 0,\text{ with}$$
\[M=\left(
\begin{array}{c|c}
0  & A \\ \hline
 B & C 
\end{array}\right),\]
where $B$ can be thought of as the matrix in the minimal free resolution of $E_{-\frac{83}{5}}$ and the residual complex is
$$\mathcal{R}_{\bullet}:\quad  \OO_{\PP^2}(-19)^9\overset{A}{\xrightarrow{\hspace{7mm}}} \OO_{\PP^2}(-18)^3\oplus \OO_{\PP^2}(-17)^2. $$
This complex is also a mapping cone according to Theorem \ref{MAINdetailed}, which makes the generalized Gaeta resolution of $U$
$$0\xrightarrow{\hspace{7mm}} \mathcal{T}_{\PP^2}(-21)^3 \xrightarrow{\hspace{7mm}} E_{-\frac{83}{5}}\oplus \mathcal{O}_{\PP^2}(-17)^2 \xrightarrow{\hspace{7mm}} U \xrightarrow{\hspace{7mm}} 0.$$ 
\end{example}

\medskip\noindent
This example exhibits that it can be difficult to guess the $F$-admissibility of a sheaf by simply looking at the Betti table of its Gaeta resolution.

\begin{rmk}\label{residual1}
The kernel $K:=ker(A)=\mathcal{H}^{-1}(\mathcal{R}_{\bullet})$ is a torsion-free sheaf of rank 4, and more importantly, is \textit{Mumford-stable}. In fact, its moduli space has dimension $78$ and it is the last model of $\PP^{2[165]}$, when the MMP is run on movable divisors. In \cite[\S 7]{Hui14}, the map ${Z}\mapsto K$ is called the Kronecker fibration.
\end{rmk}

\medskip
\subsection{The negative case: $\chi(E_{-(\alpha\cdot \beta)}, U)<0$.}\label{Negativecase}
In this subsection, we outline the proof of Theorem \ref{MAINdetailed} Part (b) as the arguments are similar to those in the previous subsection.

\begin{proof}[Proof of Theorem \ref{MAINdetailed} Part (b)] 
Let $U$ be a sheaf in $M(\xi)$ with controlling exceptional bundle $E_{-\alpha\cdot \beta}$. Set 
$m_1=-\chi(E_{-(\alpha\cdot\beta)},U)$,  $m_2=\chi(E_{-\beta},U)$, and $m_3=\chi(E_{-((\alpha\cdot\beta)\cdot\beta)},U)$.

We aim to show that if $\chi(E_{-(\alpha\cdot\beta)},U)<0$, then $U$   admits a resolution as in  \eqref{res2} if and only if it is $E_{-\alpha.\beta-3}^{m_1}$-residual and the residual complex is an $(E_{-\beta}^{m_2},E_{-\alpha-3}^{m_3})$-cone.

The proof is very similar that of Theorem \ref{MAINdetailed}, Part (a). We outline it below and omit the details as no difficulties arise. 

Let us argue the sufficiency direction. Assume $U\in M(\xi)$ lies in the Gaeta triangle 
\[  W_\bullet \xrightarrow{\hspace{7mm}} U\xrightarrow{\hspace{7mm}} F\xrightarrow{\hspace{7mm}} \cdot \]
induced by \eqref{res2} where
$F$ is the sheaf $E_{-(\alpha \cdot \beta)-3}^{m_1}$ considered in degree $-1$ and $W_\bullet$ is the following 2-term complex in the derived category $D^b(\PP^2)$ in degrees $-1$ and $0$, \[W_{\bullet}: \quad E_{-3-\alpha}^{m_3} \xrightarrow{\hspace{5mm}} E_{-\beta}^{m_2}.\]
We now proceed analogously to the proof of Theorem \ref{MAINdetailed}, Part (a) in the following steps.
\begin{enumerate}
    \item 
    The Gaeta resolutions yields line-bundle complexes $H_\bullet$ and $G[-1]_\bullet$ (in degrees $-1,0$ and $0,1$ respectively),  equivalent to $E_{-\beta}^{m_2}$ and $E_{-3-\alpha}^{m_3}$ sitting in degree $0$, respectively. 
    \item Lemma \ref{derivedHom} lifts the map in $W_\bullet$ to a map  $G[-1]_\bullet \xrightarrow{\hspace{3mm}} H_\bullet$ in the category of complexes, which realizes $W_\bullet$ as an $(E_{-\beta}^{m_2},E_{-\alpha-3}^{m_3})$-cone sitting in degrees $-1,0$.
    \item The Gaeta resolution of $E_{\alpha.\beta+3}^{m_1}$ yields a line-bundle complex  $F_\bullet$ sitting in degrees $-1,0$ equivalent to $E_{-\alpha.\beta-3}^{m_1}$ sitting in degree $-1$.
    \item Lemma \ref{derivedHom} lifts the map $F[-1]_\bullet \xrightarrow{\hspace{3mm}} W_\bullet$ to a map between complexes of line bundles in the category of complexes. The mapping cone of this map is a complex of line bundles $U_\bullet$ equivalent to $U$ in the derived category.
    \item By direct computation, $U_\bullet$ is the minimal free resolution of $U$, so it is $E_{-(\alpha \cdot \beta)-3}^{m_1}$-residual.
\end{enumerate}
The proof of the converse direction similarly follows the proof of Theorem \ref{MAINdetailed}, Part (a). 
\begin{enumerate}
    \item Being $E_{-(\alpha \cdot \beta)-3}^{m_1}$-residual yields a triangle $W \xrightarrow{\hspace{3mm}} U \xrightarrow{\hspace{3mm}} E_{-\alpha\cdot\beta-3}^{m_1}[-1] \xrightarrow{\hspace{3mm}} \cdot$. 
    \item Since $W$ is an $(E_{-\beta}^{m_2},E_{-\alpha-3}^{m_3})$-cone , we can replace $W$ by the complex $E_{-3-\alpha}^{m_3} \xrightarrow{\hspace{3mm}} E_{-\beta}^{m_2}$. 
    \item Finally, this map of complexes lifts to Resolution \eqref{res2}.
\end{enumerate}
\end{proof}

\subsection{Syzygies and base locus} We can now characterize the minimal free resolutions of sheaves in the stable base locus of the primary extremal ray of $\EFF(M(\xi))$. First, we need the following definition.

\begin{definition}\label{DDGaeta}
Let $\mathcal{U}$ denote the set in $M(\xi)$ which parametrizes sheaves that fit into a generalized Gaeta resolution. A sheaf $U\in \mathcal{U}$ is called \textit{Gaeta general.}
The complement of $\mathcal{U}$ is denoted $\mathcal{U}^c$. 
\end{definition}

\medskip\noindent
Theorem \ref{MAINdetailed} implies that $\mathcal{U}$ is an open set and consequently $\mathcal{U}^c$ has codimension at least 1. 

\begin{cor}\label{cor: EffRigid}
Let $M(\xi)$ be a nonempty moduli space of Picard rank $2$, and let $E_{\alpha\cdot \beta}$ be the controlling exceptional bundle of $\xi$.
\medskip
\begin{itemize}
    \item[(a)]  If $D_V$ spans the primary extremal ray of $\EFF(M(\xi))$, then its stable base locus satisfies
    $\mathrm{\textbf{B}}({D_V})\subset \mathcal{U}^c$. 
        \item[(b)] If $\chi(E_{-(\alpha\cdot \beta)},U)=0$, then $\mathcal{U}^c\subset M(\xi)$ is an irreducible, reduced divisor which spans an extremal ray of $\mathrm{Eff}(M(\xi))$.  
Moreover, it is the stable base locus of the primary extremal chamber of $\mathrm{Eff}(M(\xi))$.

\end{itemize}
\end{cor}

\begin{proof}
(a) It suffices to show that if a sheaf $U$ is Gaeta general, then it is not in the base locus of $D_V$, where $D_V$ spans the primary extremal ray of $\EFF(M(\xi))$. Clearly, if $U$ is Gaeta general, then it admits a resolution as in \eqref{res1} or \eqref{res2}, hence it is in the complement of some Brill-Noether divisor with class $D_V$. 
Therefore, it is not in the stable base locus of $D_V$.

\medskip\noindent
(b) By Theorem \ref{MAINdetailed}, Part (c), the locus $\mathcal{U}^c$ coincides with the Brill-Noether divisor $D_{E_{\alpha.\beta}}$, which spans the primary edge of the effective cone $\EFF(M(\xi))$ \cite[Proposition 5.9]{CHW} and is irreducible and reduced \cite[Theorem 7.3]{CHW}.

By \cite[Theorem 6.4]{CHW}, the map to the moduli space of Kronecker modules, which has Picard rank 1, is birational. Hence, it contracts a divisor which must be the base locus of the primary extremal chamber of $\EFF(M(\xi))$. This base locus must be $\mathcal{U}^c$ since it is extremal in the effective cone and irreducible.
\end{proof}

\medskip\noindent
The inclusion in Part (a) of the previous corollary is strict. For example, consider the Hilbert scheme of 9 points on the plane. The locus of 9 points which are a complete intersection of two cubics is contained in $\mathcal{U}^c$ but fails to be in $\textbf{B}({D_V})$. 

\medskip

\section{Proof of Theorem \ref{theoremB}}
\label{sec: mov}

\noindent
In this section, we provide a new computation of the cone of movable divisors $\MOV(\PP^{2[n]})$ following the \textit{base locus decomposition program} when the Gaeta resolution is pure. This is the only case in which the general element of the nontrivial extremal divisor of $\EFF(\PP^{2[n]})$ is has different Betti numbers than the generic sheaf in $\PP^{2[n]}$; that is our motivation for revisiting it. 

\medskip\noindent
Recall that an integer of the form $n=\binom{d+1}{2}$ or $n=4\binom{d+1}{2}$ is called a \textit{triangular} or \textit{tangential} number, respectively.
We first show these are the only cases when the Gaeta resolution is pure.

\begin{prop}
Let $Z\in \PP^{2[n]}$ be general. Then, the minimal free resolution of $\mathcal{I}_Z$ is pure if and only if
$n$ is a triangular or tangential number. In this case, the generalized Gaeta resolution coincides with the Gaeta resolution.
\end{prop}

\begin{proof}
Let $n$ be a positive integer, and let $d$ be the minimal integer such that $n$ general points lie on a curve of degree $d$. 
By Riemann-Roch, $d$ is the unique integer such that
$$\binom{d+1}{2}\leq n<\binom{d+2}{2}.$$
The Beilinson spectral sequence with the triad $\{\mathcal{O}_{\PP^2}(-d-2),\mathcal{O}_{\PP^2}(-d-1),\mathcal{O}_{\PP^2}(-d)\}$
yields the minimal free resolution of a general sheaf $Z \in \mathbb{P}^{2[n]}$ which has one of the forms 
\[0 \xrightarrow{\hspace{7mm}} \mathcal{O}_{\PP^2}(-d-2)^{-n_3} \oplus \mathcal{O}_{\PP^2}(-d-1)^{n_2}\xrightarrow{\hspace{7mm}} \mathcal{O}_{\PP^2}(-d)^{n_1} \xrightarrow{\hspace{7mm}} \mathcal{I}_Z \xrightarrow{\hspace{7mm}} 0\text{ or} \]

\[0 \xrightarrow{\hspace{7mm}}\mathcal{O}_{\PP^2}(-d-2)^{-n_3} \xrightarrow{\hspace{7mm}} \mathcal{O}_{\PP^2}(-d-1)^{-n_2} \oplus \mathcal{O}_{\PP^2}(-d)^{n_1} \xrightarrow{\hspace{7mm}} \mathcal{I}_Z \xrightarrow{\hspace{7mm}} 0,\]
where $$n_3=\chi(\mathcal{O}_{\PP^2}(-d+1),\mathcal{I}_Z),\quad n_2=\chi(\mathcal{T}(-d-1),\mathcal{I}_Z), \quad\text{and} \quad  n_1=\chi(\mathcal{O}_{\PP^2}(-d),\mathcal{I}_Z).$$

\medskip\noindent
Therefore, the Gaeta resolution above is pure if and only if some $n_i=0$. Using Riemann-Roch, we see that $n_1=\binom{d+2}{2}-n>0$ is never zero. Moreover:
\begin{enumerate}[(1)]
    \item $n_2=0$ if and only if $d=2d'$ and $n=4\binom{d'+1}{2}$, which is a tangential number.
    \item $n_3=0$ if and only if $n=\binom{d+1}{2}$, which is a triangular number.
\end{enumerate}
Finally, when $n$ is a triangular number, the exceptional collection above yields the generalized Gaeta resolution. When $n$ is a tangential number, the exceptional collection that yields the generalized Gaeta resolution is $\{\OO_{\PP^2}(-d-2), \OO_{\PP^2}(-d), \mathcal{T}_{\PP^2}(-d-1)\}$ and the power of $\mathcal{T}_{\PP^2}(-d-1)$ is 0. Therefore, the generalized Gaeta resolution coincides with the Gaeta resolution in both cases.
\end{proof}

\medskip\noindent
\begin{remark}\label{DivGaeta}
The previous proposition and Corollary \ref{cor: EffRigid} imply that if $n$ is a triangular or tangential number, then $\mathcal{U}^c$ (Definition \ref{DDGaeta}) is an irreducible divisor. Furthermore, there exists a unique \textit{divisorial} Betti table. That is, the sheaves with that Betti table form an open set of $\mathcal{U}^c$. Indeed,
if $d>2$ is a triangular number, then the general sheaf $Z\in \mathcal{U}^c$ has minimal free resolution 
\begin{equation*} 
0\xrightarrow{\hspace{7mm}}\OO_{\PP^2}(-d-2)\oplus\OO_{\PP^2}(-d-1)^{d-3}\xrightarrow{\hspace{7mm}}\OO_{\PP^2}(-d)^{d-2}\oplus\OO_{\PP^2}(-d+1)\xrightarrow{\hspace{7mm}} \mathcal{I}_Z \xrightarrow{\hspace{7mm}} 0.
\end{equation*}
A dimension count (see \cite{BS}) shows that the family of sheaves with this resolution has dimension $2n-1$.
If $n$ is a tangential number, then the general sheaf $\mathcal{I}_Z\in\mathcal{U}^c$ has minimal free resolution
\begin{equation}\label{eq::dBettiTangential}
0\xrightarrow{\hspace{7mm}}\OO_{\PP^2}(-2d-2)^d\oplus\OO_{\PP^2}(-2d-1)\xrightarrow{\hspace{7mm}}\OO_{\PP^2}(-2d-1)\oplus\OO_{\PP^2}(-2d)^{d+1}\xrightarrow{\hspace{7mm}} \mathcal{I}_Z \xrightarrow{\hspace{7mm}} 0.
\end{equation}
    Again, a dimension count shows that the family of sheaves with this resolution has dimension $2n-1$. Summarizing, if $Z\in \PP^{2[n]}$ is a general element in the divisor $\mathcal{U}^c$, where $n$ is a triangular or tangential number respectively, then its Betti table is:
$$
\begin{array}{c|cc}
&1&2 \\
\hline 
 d-2&\mathbf{1}& \\
 d-1&d-2&d-3\\
 d& &\mathbf{1}
 \end{array}\quad  \quad  \quad 
 \begin{array}{c|cc}
&1&2 \\
\hline 
 2d-1&d+1&\mathbf{1}\\
 2d&\mathbf{1}&d\\
 \end{array}
$$
\end{remark}

\begin{definition}
  If $n$ is a triangular or a tangential number, then we denote $D_{\tiny{Betti}}:=\mathcal{U}^c\subset \PP^{2[n]}$.
\end{definition}

\medskip\noindent
It follows from \cite[Theorem E]{LZ} that the destabilizing object that yields the primary extremal wall of $\MOV(\PP^{2[n]})$ is $\OO_{\PP^2}(-d+1)$, in the triangular case $n=\binom{d+1}{2}$, and $\mathcal{T}_{\PP^2}(-2d-1)$, in the tangential case $n=4\binom{d+1}{2}$. The minimal free resolutions of both destabilizing objects are contained in the respective minimal free resolutions of a general $Z\in D_{Betti}$ above. Summarizing:

\begin{prop}\label{prop::destObjMov}
    The minimal free resolution of a general $Z\in D_{Betti}$ contains the minimal free resolution of the Bridgeland destabilizing objects that induce the extremal wall of the movable cone.
\end{prop}
\begin{proof}
    The triangular case is trivial. For the tangential case, the part of Resolution \eqref{eq::dBettiTangential} corresponding to the morphism
    $$\OO_{\PP^2}(-2d-1)\xrightarrow{\hspace{5mm}}\OO_{\PP^2}(-2d-1)\oplus\OO_{\PP^2}(-2d)^{d+1}$$
    is zero in the first factor, and given by $d+1\geq3$ linear forms in the second factor. If $Z$ is general in $D_{Betti}$, these linear forms can be brought, after a series of row operations, to a basis of the space of linear forms in the last three entries; and zero otherwise. This yields the resolution of $\mathcal{T}_{\PP^2}(-2d-1)$.
\end{proof}

\medskip\noindent
This proves the second part of Theorem \ref{theoremB}. Let us now show that the syzygies of a general element in $D_{Betti}$ determine the extremal wall of $\MOV(\PP^{2[n]})$.

\subsection{Movable cone of $\PP^{2[n]}$ for $n$ a triangular number.}  This subsection proves Theorem \ref{TRI} below, which concludes the proof of Part (a) of Theorem \ref{theoremB}. 

\medskip\noindent
Our description of the non-trivial extremal divisor of the movable cone consists of two steps. First, we solve the interpolation problem for a general element in $D_{\tiny{Betti}}$. Second, we show that the induced Brill-Noether divisor in fact spans an extremal ray of $\MOV(\PP^{2[n]})$. We start by addressing the first step, which is the difficult one.

\medskip\noindent
Proposition \ref{prop::destObjMov} asserts that the destabilizing object $\OO_{\PP^2}(-d+1)$ is contained in the minimal free resolution of a general $Z\in D_{\tiny{Betti}}$. This yields a distinguished triangle $\OO_{\PP^2}(-d+1)\xrightarrow{\hspace{3mm}} \mathcal{I}_Z\xrightarrow{\hspace{3mm}}W_\bullet\xrightarrow{\hspace{3mm}}\cdot$ where $W_\bullet[-1]$ is the residual complex. Using the procedure described after Definition \ref{orthogonal}, we find the invariants of a vector bundle $M$ numerically orthogonal to both $\mathcal{I}_Z$ and $\OO_{\PP^2}(-d+1)$, which ultimately solves the interpolation problem for $\mathcal{I}_Z$. First, we show that $M$ is in fact cohomologically orthogonal to $\mathcal{I}_Z$.

\begin{lem}\label{InterTri}
Let $n=\binom{d+1}{2}$ be a triangular number with $d>1$. 
The general bundle with resolution \[0\xrightarrow{\hspace{7mm}} \mathcal{O}_{\PP^2}(d-3)^{kd}\xrightarrow{\hspace{7mm}} \mathcal{O}_{\PP^2}(d-2)^{k(2d-1)}\xrightarrow{\hspace{7mm}} M\xrightarrow{\hspace{7mm}} 0\]
is cohomologically orthogonal to the general point in $D_{\tiny{Betti}}\subset \PP^{2[n]}$ if $k$ is sufficiently large.
\end{lem}

\begin{proof} 
Let $Z\in \PP^{2[n]}$ and observe that $\chi(M\otimes \mathcal{I}_Z)=0$ by Riemann-Roch. We show the vanishing of the cohomology groups $h^i(M\otimes \mathcal{I}_Z)$, $i=0,1,2$, when $Z\in D_{\tiny{Betti}}$, by reducing the problem to a computation on an irreducible plane rational curve of degree $d-1$.
Given a curve $C\subset \PP^2$ of degree $d-1$, there is an exact sequence
\[0\xrightarrow{\hspace{7mm}} M(-d+1)\xrightarrow{\hspace{7mm}} M\xrightarrow{\hspace{7mm}} M|_C\xrightarrow{\hspace{7mm}}0.\]
From the minimal free resolution of $M$ above, we see that $H^0(\mathbb{P}^2,M(-d+1))=H^1(\mathbb{P}^2,M(-d+1))=0$. 
This implies that the restriction morphism
\[H^0(\mathbb{P}^2,M)\xrightarrow{\hspace{7mm}} H^0(C,M|_C)\]
is an isomorphism. 
Let $C$ be the image of a general map $f:\mathbb{P}^1\xrightarrow{\hspace{3mm}}\mathbb{P}^2$ of degree $d-1$. 
Applying \cite[Theorem 1.8]{Hui13} on $M(-d+1)$, noting that 
$\frac{d}{d-1}\in\Phi_2$, we have that 
\[f^*M\cong\mathcal{O}_{\mathbb{P}^1}(d^2-2d+2)^{k(d-1)}\]
for any sufficiently large $k$.

The curve $C$ has exactly $\binom{d-2}{2}$ nodes as $f$ is general.  Let $Z\subset C$ be the divisor that consists of these nodes in addition to $n-\binom{d-2}{2}=3d-3$ general points of $C$ (so $Z\in D_{\tiny{Betti}}$).  Then, we have
\[\deg f^*Z=2\binom{d-2}{2}+3d-3=d^2-2d+3.\]
Observe that a section of $M|_C$ vanishing on $D$ induces a section of $f^*M$ vanishing on $f^*Z$ and vice versa. In other words,
\begin{align*}
    H^0(C,M|_C(-Z))&\cong H^0(\mathbb{P}^1,(f^*M)(-f^*Z))\\
    &\cong H^0(\mathbb{P}^1,\mathcal{O}_{\mathbb{P}^1}(-1))^{k(d-1)}=0.
\end{align*}
It follows that $H^0(\mathbb{P}^2,M\otimes\mathcal{I}_Z)\cong H^0(C,M|_C(-Z))=0$.

Notice that by tensoring the short exact sequence \[0 \xrightarrow{\hspace{7mm}} \mathcal{I}_Z \xrightarrow{\hspace{7mm}}\mathcal{O}_{\PP^2}\xrightarrow{\hspace{7mm}} \mathcal{O}_Z \xrightarrow{\hspace{7mm}} 0\] with $M$, it follows that $H^2(\mathbb{P}^2,M\otimes\mathcal{I}_Z) = 0$. 
Therefore, $M$ is cohomologically orthogonal to $\mathcal{I}_Z$.
Since $D_{\tiny{Betti}}$ is irreducible, $M$ is also cohomologically orthogonal to the general element in it.
\end{proof}

\medskip\noindent
Let us now use the previous bundle $M$ to describe the movable cone $\mathrm{Mov}\left(\mathbb{P}^{2[n]}\right)$. Recall that the nontrivial edge of $\EFF(\PP^{2[n]})$ has class $(d-1)H-\frac{1}{2}B$.

\begin{theorem}\label{TRI} Let $n=\binom{d+1}{2}$ be a triangular number with $d>2$. The primary extremal edge of the movable cone $\MOV(\PP^{2[n]})$ is spanned by the divisor class  \[D_{\mathrm{mov}} = \tfrac{d^2 - 2 d + 2}{d - 1}H-\tfrac{1}{2}B= \left(d-1+\tfrac{1}{d - 1}\right)H-\tfrac{1}{2}B .\]
\end{theorem}
\begin{proof}
    First note that the bundle $M$ in Lemma \ref{InterTri} has slope
    $$\mu(M)=\frac{d^2-2d+2}{d-1}.$$
    Therefore, the class of the Brill-Noether divisor defined by $M$ satisfies $D_M=k(d-1)\cdot D_{\mathrm{mov}}.$ Again by Lemma \ref{InterTri}, the general element $Z\in D_{\tiny{Betti}}$ is not contained in the base locus of $D_M$. Therefore, $D_{\mathrm{mov}}$ is movable.

    Now consider a divisor class $\mu H-B$ with $\mu<\mu(M)$. A general pencil of $n$ points on a plane curve of degree $d-1$ induces a curve $\alpha\subset\PP^{2[n]}$ whose class sweeps out $D_{\tiny{Betti}}$. It is easy to see that $D_M\cdot\alpha=0$. In particular, $(\mu H-B)\cdot\alpha<0$, so $D_{\tiny{Betti}}$ is contained in the stable base locus of $\mu H-B$. In particular, $\mu H-B$ is not a movable class.
\end{proof}

\medskip\noindent
The following corollary exhibits the pattern followed by the \textit{base locus decomposition program.}

\begin{cor}\label{SolveInterTri}
The bundle $M$ of Lemma \ref{InterTri} has the minimal slope of any semi-stable bundle cohomologically orthogonal to the ideal sheaf of a generic $Z\in D_{\tiny{Betti}}\subset \PP^{2[n]}$, where $n=\binom{d+1}{2}$, $d>2$. In other words, $M$ solves the interpolation problem for the general $Z\in D_{\tiny{Betti}}$.
\end{cor}
\begin{proof}
    If $M'$ is cohomologically orthogonal to the general $Z\in D_{\tiny{Betti}}$ then $D_{M'}$ must be a movable class, so $\mu(M')\geq\mu(M)$ by Theorem \ref{TRI}.
\end{proof}

\subsection{Movable cone of $\PP^{2[n]}$ for $n$ a tangential number.} This subsection proves Theorem \ref{MOV} below, which concludes the proof of Part (b) of Theorem \ref{theoremB}.

\medskip\noindent
This case presents two additional difficulties to the previous subsection.
First, showing cohomological orthogonality in this case is more difficult. 
In fact, rather than showing orthogonality for the general element of $D_{\tiny{Betti}}$ as before, we do so for special elements and then argue by upper-semicontinuity.
Second, the construction of the curve class dual to the edge of the movable cone is more complicated (Proposition \ref{Prop:: InterpoTang}) --- albeit with interesting geometry.

\medskip\noindent 
For $n=2d(d+1)$, according to \cite[Theorem 3.13]{E}, there exist ideal sheaves $\mathcal{I}_Z$ of $n$ points with minimal free resolution
\begin{align}\label{qk}
0\xrightarrow{\hspace{7mm}}\OO_{\PP^2}(-2d-2)^d\oplus\OO_{\PP^2}(-2d-1)^k\overset{\phi}{\xrightarrow{\hspace{7mm}}}\OO_{\PP^2}(-2d-1)^k\oplus\OO_{\PP^2}(-2d)^{d+1}\xrightarrow{\hspace{7mm}} \mathcal{I}_Z \xrightarrow{\hspace{7mm}} 0,
\end{align}
as long as $k\leq d$. 
The minimal free resolution of a general $Z\in D_{\tiny{Betti}}$ is \eqref{qk} for $k=1$, which yields:
\begin{align}\label{qk1}
\mathcal{T}_{\mathbb{P}^2}(-2d-1)_{\bullet} \xrightarrow{\hspace{7mm}} (\mathcal{I}_Z)_{\bullet} \xrightarrow{\hspace{7mm}} W_{\bullet}\xrightarrow{\hspace{7mm}} \cdot 
\end{align}

\medskip\noindent
We show below that this triangle determines the movable cone $\MOV(\PP^{2[n]})$, without appealing to any Bridgeland stability results.
The sheaves with Resolution \eqref{qk} are in $D_{\tiny{Betti}}$ by Remark \ref{DivGaeta}. 
In fact, when $k=2$, these are the sheaves for which we establish cohomological orthogonality.
Recall that by Corollary \ref{cor: EffRigid} and \cite{CHW}, $D_{\tiny{Betti}}$ parametrizes sheaves satisfying $h^0(\PP^2,\mathcal{T}_{\PP^2}(2d-2)\otimes\mathcal{I}_Z) =\mathrm{Hom}(\mathcal{T}_{\PP^2}(-2d-1),\mathcal{I}_Z) \neq 0$.

\medskip\noindent
The following lemma is used in Proposition 4.11 and tells us the number of conditions that a scheme with Resolution \eqref{qk} imposes on sections of the (twisted) tangent bundle of $\PP^2$. We need a weaker version of it ($k=1,2$) but the general case is not much harder to show.

\begin{lem}\label{lem: qk sections}
A scheme $Z$ with Resolution \eqref{qk} satisfies $h^0(\PP^2,\mathcal{T}_{\PP^2}(2d-2)\otimes\mathcal{I}_Z) = k$.
\end{lem}
\begin{proof}
Since the twisted tangent sheaf $\mathcal{T}_{\PP^2}(m)$ has no cohomology for $m=-2,-4$,  we get the following exact sequence after tensoring \eqref{qk} with $\mathcal{T}_{\PP^2}(2d-2)$ and taking cohomology: 
\[0\xrightarrow{\hspace{7mm}} H^0(\PP^2,\mathcal{T}_{\PP^2}(2d-2)\otimes\mathcal{I}_Z)\xrightarrow{\hspace{7mm}} H^1(\PP^2,\mathcal{T}_{\PP^2}(-3))^k\overset{\phi_*}{\xrightarrow{\hspace{7mm}}}H^1(\PP^2,\mathcal{T}_{\PP^2}(-3))^k.\]

The morphism $\phi_*$ is induced by $\phi$ in \eqref{qk} and is the zero map since \eqref{qk} is minimal. 
Therefore, \[h^0(\PP^2,\mathcal{T}_{\PP^2}(2d-2)\otimes\mathcal{I}_Z)=k\cdot h^1(\PP^2,\mathcal{T}_{\PP^2}(-3))=k.\]
\end{proof}

\begin{rmk}\label{folations}
Resolution \eqref{qk}, with $k=1$, yields the minimal free resolution of a subscheme $Z$ that uniquely determines a section of the tangent sheaf $\mathcal{T}_{\PP^2}(2d-2)$. Observe that the length $|Z|=2d(d+1)$ is strictly less than the length of the zero locus of a section; which is computed by the second Chern class $c_2(\mathcal{T}_{\PP^2}(2d-2))$. In fact, the length of $Z$ is the minimal possible length a subscheme that determines a section of $\mathcal{T}_{\PP^2}(2d-2)$ can have. Hence, $Z$ answers in this setting the Campillo-Olivares problem about foliations on $\mathbb{P}^2$ which are completely determined by a subscheme of minimal length of its singular locus \cite{CO,CO1,O}. We thank J. Olivares for pointing that out to us.

\medskip\noindent
Moreover, in Resolution \eqref{qk1} the residual complex $W$ carries the information of $\mathcal{O}_R$, where $R$ is the residual subscheme to $Z$ in the zero locus of the unique section of $\mathcal{T}_{\PP^2}(2d-2)$ that $Z$ determines.
\end{rmk}

\medskip\noindent
We now start solving the interpolation problem for a general point in $D_{\tiny{Betti}}$ using the previous lemma.

\begin{lem}\label{InterpolationTan}
Let $n=2d(d+1)$ be a tangential number.
Then the general bundle with the resolution
\[0\xrightarrow{\hspace{7mm}} \mathcal{O}_{\PP^2}(2d-3)^{kd}\xrightarrow{\hspace{7mm}} \mathcal{O}_{\PP^2}(2d-1)^{k(5d-1)}\xrightarrow{\hspace{7mm}} M\xrightarrow{\hspace{7mm}} 0\]
is cohomologically orthogonal to the general point in $D_{\tiny{Betti}}$ for $k$ sufficiently large.
\end{lem}

\begin{proof}
Since $D_{\tiny{Betti}}$ is an irreducible divisor, it suffices to provide a single point in $D_{\tiny{Betti}}$ whose ideal sheaf is cohomologically orthogonal to the bundle $M$. We do this next.

For $1<d<6$, cohomological orthogonality can be proven using Macaulay2.
Appendix A contains code in order to do so.
Let us assume $d\geq 6$. 
There are ideal sheaves $\mathcal{I}_Z$ with resolution 
\begin{equation}\label{SPECIAL2}
0 \xrightarrow{\hspace{7mm}} \mathcal{O}_{\PP^2}(-2d-2)^d \oplus \mathcal{O}_{\PP^2}(-2d-1)^2 \overset{\phi}{\xrightarrow{\hspace{7mm}}} \mathcal{O}_{\PP^2}(-2d-1)^2 \oplus \mathcal{O}_{\PP^2}(-2d)^{d+1},    
\end{equation}
which is \eqref{qk} with $k=2$. These are contained in $D_{\tiny{Betti}}$ as they do not have a Gaeta resolution.
There exist row and column reductions which reduce the matrix for $\phi$ to the form
\[\phi=\left(\begin{array}{c|c}
0    & A \\ \hline
B & C
\end{array}\right),\]
where $B$ is the standard matrix for the resolution of $\mathcal{T}_{\mathbb{P}^2}(-2d-1)^2$.
Since we only need a single sheaf to be cohomologically orthogonal, we may assume that $A$ is general and so defines a vector bundle $V$ by
\begin{align}\label{eq: esp res}0 \xrightarrow{\hspace{7mm}} V \xrightarrow{\hspace{7mm}} \mathcal{O}_{\PP^2}(-2d-2)^d \overset{A}{\xrightarrow{\hspace{7mm}}} \mathcal{O}_{\PP^2}(-2d-1)^2 \oplus \mathcal{O}_{\PP^2}(-2d)^{d-5}\xrightarrow{\hspace{7mm}} 0.\end{align}

In other words, the minimal free resolution of $\mathcal{I}_Z$, when considered in the derived category, is equivalent to the triangle
\[\mathcal{T}_{\mathbb{P}^2}(-2d-1)^2_\bullet \xrightarrow{\hspace{7mm}} (\mathcal{I}_Z)_{\bullet} \xrightarrow{\hspace{7mm}} V_{\bullet}\xrightarrow{\hspace{7mm}} \cdot \] where $V_\bullet$ is the sheaf $V$ considered\footnote{$V_\bullet$ is concentrated in degree -1 so that it is the Bridgeland destabilizing quotient.} in degree $-1$.
The cohomology of this triangle yields the short exact sequence of sheaves
\[0\xrightarrow{\hspace{7mm}} V\xrightarrow{\hspace{7mm}} \mathcal{T}_{\mathbb{P}^2}(-2d-1)^2 \xrightarrow{\hspace{7mm}} \mathcal{I}_Z \xrightarrow{\hspace{7mm}} 0.\]

It now suffices to show that $M$ is cohomologically orthogonal to $\mathcal{T}_{\mathbb{P}^2}(-2d-1)^2$ as well as to a general $V$. Since $\chi(M\otimes \mathcal{T}_{\mathbb{P}^2}(-2d-1)) = 0$, then $\mathcal{T}_{\mathbb{P}^2}(-2d-1)$ is cohomologically orthogonal to a general $M$ as it is an exceptional bundle \cite{CHK}.  
This also can be seen directly from the resolution of $M$ since $\mathcal{T}_{\mathbb{P}^2}(-2)$ and $\mathcal{T}_{\mathbb{P}^2}(-4)$ have no cohomology.

To show $h^i(M\otimes V)=0$, with $i=0,1,2$, we apply \cite[Thm. 1.2]{CHK}. 
To do this, $V$ must be a general stable bundle in a positive dimensional moduli space. This is indeed the case as $A$ is general and $V^*$ has the Gaeta minimal free resolution. Furthermore, the discriminant $\Delta(V) = \frac{2d^2-10d+5}{9}$ is above the Dr\'ezet-Le Potier curve for $d\geq 1$ which makes it stable \cite{DL}. Thus, $V$ and $M$ satisfy the hypotheses of \cite[Thm. 1.2]{CHK}.

In order to see which case of \cite[Thm. 1.2]{CHK} we fall into, we need to compute the primary controlling exceptional bundle to $V$.
Recall that the primary controlling exceptional bundle associated to $V$ is the exceptional bundle whose left and right endpoint Chern characters form an interval that contains the largest solution to the equations:
\[\chi(V \otimes F)=0 \quad \text{ and }\quad  \Delta(F) = \frac{1}{2}.\] 
Since $\mathrm{rk}(V) = 3$, $\mu(V)  = \frac{2-8d}{3}$ and $\Delta(V) = \frac{2d^2-10d+5}{9}$, solving these equations for the slope of $F$ gives
\[\mu(F) = \frac{1}{6} \left(\sqrt{16 d^2 - 80 d + 85} + 16 d - 13\right).\]
If $d$ is congruent to $0 \mod 3$, the left and right endpoints of the interval controlled by $\mathcal{O}_{\PP^2}(\frac{10}{3}d-4)$ are $\frac{1}{6}(-33 + 3 \sqrt{5} + 20 d)$ and $\frac{1}{6} (20 d - 3 (5 + \sqrt{5}))$, respectively.
The slope $\mu(F)$ is between those endpoints if and only if $d \geq 4 = \left\lceil \frac{15 \sqrt{5} - 45}{3 \sqrt{5} - 10}\right\rceil$.
If $d$ is congruent to $1\mod 3$ ($2 \mod 3$), a similar process shows that the primary controlling exceptional bundle is $E_{\frac{1}{2}}(\frac{10d-13}{3})$ (or $\mathcal{O}_{\PP^2}(\frac{10d-11}{3})$).  
As a result, we have to address each of these cases separately.

First, if $d$ is congruent to $0$ mod $3$, then we have $\chi\left(V\left(\frac{10}{3}d-4\right)\right) = \frac{9 - d}{3}<0$.
By Part 1 of Theorem 1.2 in \cite{CHK}, the cohomology of $M\otimes V$ is determined by $\chi(M\otimes V) = 0$.

Next, if $d$ is congruent to $1$ mod $3$, then we have $\chi\left(V\otimes E_{\frac{1}{2}}\left(\frac{10d-13}{3}\right)\right) = 2 > 0$.
Since \[\mathrm{rank}(V) - \chi\left(V \otimes E_{\frac{1}{2}}\left(\frac{10d-13}{3}\right)\right) \mathrm{rank}\left(E_{\frac{1}{2}}\left(\frac{10d-13}{3}\right)^*\right)= 3-2\cdot2 = -1<0,\] the cohomology of $M\otimes V$ is determined by $\chi(M\otimes V) = 0$, by Part 2 of Theorem 1.2 in \cite{CHK}.

Finally, if $d$ is congruent to $2$ mod $3$, then we have $\chi\left(V\left(\frac{10d-11}{3}\right)\right) = \frac{d+4}{3}>0$.
Since \[\mathrm{rank}(V) - \chi\left(V\left(\frac{10d-11}{3}\right)\right)\mathrm{rank}\left(\mathcal{O}_{\PP^2}\left(\frac{11-10d}{3}\right)\right)= 3-1\cdot(d+4) = -1-d<0\] 
the cohomology of $M\otimes V$ is determined by $\chi(M\otimes V) = 0$, by Part 2 of Theorem 1.2 in \cite{CHK}.

In every case, the bundle $M$ is cohomologically orthogonal to $V$ and hence to a general $\mathcal{I}_Z$ in $D_{\tiny{Betti}}$.
\end{proof}

\medskip\noindent
The following proposition exhibits the solution to the interpolation problem for a general point in $D_{\tiny{Betti}}$. 

\begin{prop}\label{Prop:: InterpoTang}
The bundle $M$ of Lemma \ref{InterpolationTan} has the minimum slope among bundles cohomologically orthogonal to the generic point in $D_{\tiny{Betti}}\subset \PP^{2[n]}$, when $n=2d(d+1), d\geq 1$.
\end{prop}
\begin{proof}
Similarly to Corollary \ref{SolveInterTri}, it suffices to exhibit a  curve class $\beta\subset \PP^{2[n]}$ that both sweeps out $D_{\tiny{Betti}}$ and is orthogonal to $D_M$. We do this next.

Consider a general element $Z\in D_{\tiny{Betti}}$ as part of the vanishing locus of a section in $H^0(\PP^2, \mathcal{T}_{\PP^2}(2d-2))$, which we denote by $\Gamma=Z\cup\Gamma_{\tiny{res}}$.
Note that $\Gamma_{\tiny{res}}$ has length
\[\delta:=c_2(\mathcal{T}_{\PP^2}(2d-2))-n=2d^2-4d+1.\]
The space of degree $4d-1$ curves which pass through $Z$ and are nodal at $\Gamma_{res}$ has dimension
\[\binom{4d-1+2}{2}-1-n-3\delta=12d-4.\]
Therefore, we may consider an irreducible curve $C$ of degree $4d-1$ which passes through $Z$ and has simple nodes at $\Gamma_{res}$ as its only singularities. Using the Riemann-Roch theorem, we get
\begin{align*}
    h^0(C,\mathcal{O}_C(Z))&=h^0(C,\omega_C \otimes \mathcal{O}_C(-Z))+n+1-g(C)\\
    &=h^0(\PP^2,\mathcal{I}_{\Gamma}(4d-4))+n+1-\binom{4d-2}{2}+\delta\\
    &=h^0(\PP^2,\mathcal{I}_{\Gamma}(4d-4))-(4d^2-8d+1).
\end{align*}
We know that $\Gamma$ has minimal free resolution as stated in Lemma \ref{lem: resolution zeroes tangent}\footnote{We defer the proof of this lemma as its proof is independent of the rest of the proof of Theorem \ref{MOV}.} which, in particular, implies that
\[h^0(\PP^2,\mathcal{I}_{\Gamma}(4d-4))=4d^2-8d+3.\]
This proves that $h^0(C,\mathcal{O}_C(Z))=2$, and hence, $Z$ moves in a pencil on $C$, which can be proved to be base point free. This pencil induces a curve class $\beta\subset\mathbb{P}^{2[n]}$ with a representative which contains $Z$ and satisfies
\[
    \beta\cdot H = 4d-1 \quad \text{and} \quad
    \beta\cdot \tfrac{1}{2}B=8d^2-4d+1.
\]
Therefore, $D_{M}\cdot\beta = 0$.
Since $Z$ was general in $D_{\tiny{Betti}}$, curves with class $\beta$ sweep out $D_{\tiny{Betti}}$, and the result follows.
\end{proof}

\medskip\noindent
Note that the assumption $d>1$ is necessary in the previous proof as otherwise $\delta<0$.

\medskip\noindent
We now compute the primary extremal divisor of the movable cone $\mathrm{Mov}\left(\mathbb{P}^{2[n]}\right)$, when $n$ is a tangential number.

\begin{theorem}\label{MOV} Let $n=2d(d+1)$, with $d\geq2$.
The primary extremal ray of the movable cone $\MOV(\PP^{2[n]})$ is spanned by the divisor class  \[D_{mov} = \tfrac{8 d^2 - 4 d + 1}{4 d - 1}H-\tfrac{1}{2}B.\]
\end{theorem}

\begin{proof} Using Remark \ref{DivGaeta} (see also Proposition \ref{prop::destObjMov}), we can write the minimal free resolution of a general $Z\in D_{\tiny{Betti}}$ in the derived category as follows 
\begin{equation*}
    \begin{aligned}
   W[-1]\xrightarrow{\hspace{7mm}} \mathcal{T}_{\mathbb{P}^2}(-2d-1)\xrightarrow{\hspace{7mm}}\mathcal{I}_Z\xrightarrow{\hspace{7mm}}   \cdot
    \end{aligned}
\end{equation*}
where $W_\bullet$ is the following complex of rank $1$ sitting in degrees $-1$ and $0$:
$$W_\bullet: \OOT(-2d-2)^d \xrightarrow{\hspace{5mm}} \OOT(-2d-1)\oplus\OOT(-2d)^{d-2}.$$

The log Chern character $\zeta$ that satisfies the equations 
\begin{equation*}
 \chi(\zeta\otimes \mathcal{T}_{\mathbb{P}^2}(-2d-1))=0\quad\text{and} \quad       
 \chi(\zeta\otimes W)=0
\end{equation*}
is that of the bundles in Lemma \ref{InterpolationTan} (See Figure 2).

By Lemma \ref{InterpolationTan} and Proposition \ref{Prop:: InterpoTang}, the general bundle with resolution \[0\xrightarrow{\hspace{7mm}} \mathcal{O}_{\PP^2}(2d-3)^{kd}\xrightarrow{\hspace{7mm}} \OOT(2d-1)^{k(5d-1)}\xrightarrow{\hspace{7mm}} M \xrightarrow{\hspace{7mm}} 0\] solves the interpolation problem for the general point in $D_{\tiny{Betti}}$. 
The Brill-Noether divisor induced by the bundle $M$, whose class is $D_{M}=\mu_M H-B/2$, does not contain $D_{\tiny{Betti}}$ in its base locus and, therefore, is movable. 
Furthermore, this class is dual to the curve class $\beta$ constructed in the proof of Proposition \ref{Prop:: InterpoTang}. 
Since this curve class sweeps out $D_{\tiny{Betti}}$, it follows that the class $D_M$ is extremal in the movable cone $\Mov(\PP^{2[n]})$.
\end{proof}

\begin{center}\label{figuretangent}
\begin{figure}[htb]
\resizebox{.75\textwidth}{!}{\includegraphics{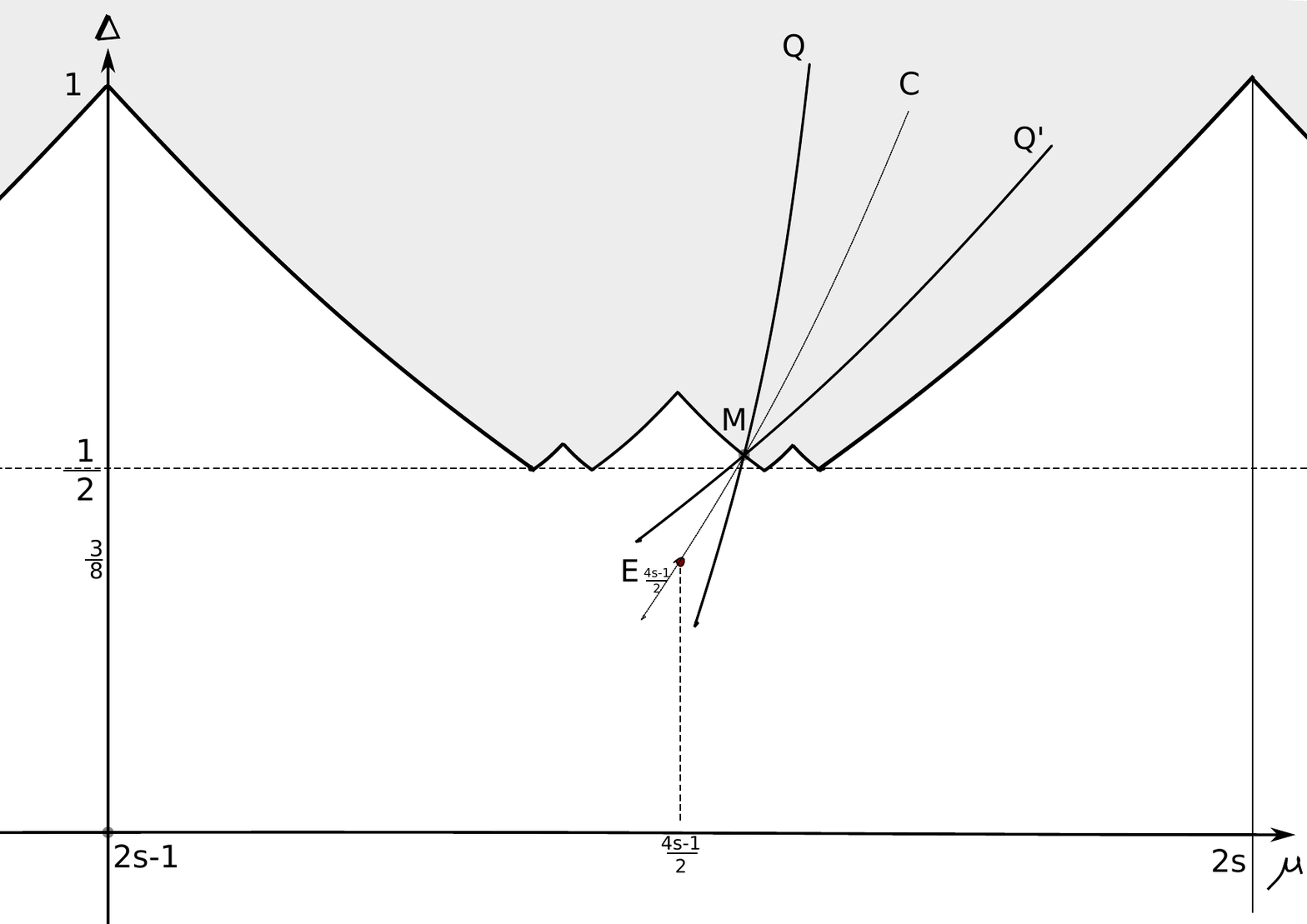}}
\caption{The invariants $(\mu(M),\Delta(M))$ of the bundle $M$ that determine the movable cone lies at the intersection of the curves $Q=\{\chi(\zeta\otimes \mathcal{T}_{\mathbb{P}^2}(-2d-1))=0\}$ and $Q'=\{\chi(\zeta\otimes W)=0\}$, where $W$ is the complex in the proof of Theorem \ref{MOV}. The curve $C$ contains the pairs $(\mu,\Delta)$ orthogonal to $ch(\mathcal{I}_Z)=(1,0,2d(d+1))$.}
\end{figure}
\end{center}

The following lemma was used in the proof of Proposition \ref{Prop:: InterpoTang}.

\begin{lem}\label{lem: resolution zeroes tangent}
    Let $\gamma\in H^0(\PP^2,\mathcal{T}_{\PP^2}(2d-2))$ be a general section, and let $\Gamma$ be the zero locus of $\gamma$. Then the ideal sheaf $\mathcal{I}_{\Gamma}$ has minimal free resolution
    \begin{equation}\label{res zero section}
        0\xrightarrow{\hspace{7mm}}\OO_{\PP^2}(-4d+1)\oplus\OO_{\PP^2}(-2d-1)\xrightarrow{\hspace{7mm}}\OO_{\PP^2}(-2d)^3\xrightarrow{\hspace{7mm}}\mathcal{I}_\Gamma\xrightarrow{\hspace{7mm}}0.
    \end{equation}
\end{lem}
\begin{proof}
    If $\Gamma$ is any set of points with the above resolution, then 
    tensoring by $\mathcal{T}_{\PP^2}(2d-2)$ and taking the long exact sequence in cohomology shows that $h^0(\PP^2,\mathcal{I}_\Gamma\otimes \mathcal{T}_{\PP^2}(2d-2))=1$.
    Therefore, the locus of ideals $\mathcal{I}_\Gamma\in \PP^{2[n]}$ with the minimal free resolution above is contained in the locus of subschemes that are the zeroes of a section of $\mathcal{T}_{\PP^2}(2d-2)$. Consequently, the lemma follows if we show that these two families have the same dimension as they are both irreducible. This is what we show next.

    Let $\Gamma$ be the zero locus of a general section $\gamma$, and let $Z\subset\Gamma$ be a subscheme of length $n = 2d(d+1)$. Then $\gamma\in H^0(\PP^2,\mathcal{I}_Z\otimes \mathcal{T}_{\PP^2}(2d-2))\neq0$, so $Z\in D_{\tiny{Betti}}$. 
    Conversely, given a general element $Z\in D_{\tiny{Betti}}$ there is exactly one section (up to scalar multiple) in $H^0(\PP^2,\mathcal{I}_Z\otimes \mathcal{T}_{\PP^2}(2d-2))$ (by Lemma \ref{lem: qk sections} with $k=1$), which defines a scheme $\Gamma\supset Z$ as its zero locus.
    This defines a rational, generically finite, and dominant map from $D_{\tiny{Betti}}$ to the locus of schemes which are zeroes of a section of $\mathcal{T}_{\PP^2}(2d-2)$, which therefore has dimension equal to $\dim D_{\tiny{Betti}}=2n-1$.

    On the other hand, by \cite{BS} as explained in \cite{CM}, the dimension of the schemes with Resolution \eqref{res zero section} can be computed as \[3\cdot\binom{4d-1-2d+2}{2}+3\cdot\binom{2d+1-2d+2}{2}-2-\binom{4d-1-(2d+1)+2}{2}-9+1=2n-1.\]
\end{proof}

\medskip

\section{Explicit Examples}\label{section5}
\noindent
We started the paper asking what data from the minimal free resolution determines the SBLD and the Bridgeland walls of a moduli space $M(\xi)$. This section exhibits examples for which we fully answer this question. In all cases the answer is: the Betti table along with the $F$-admissibility of the minimal free resolution. 

\medskip\noindent
First, we run the \textit{stable base locus decomposition program} and describe the walls it produces in the stable base locus decomposition \footnote{We omit the unique chamber containing the component $B$ of non-reduced schemes.} of $\EFF(\PP^{2[n]})$ when $n=3,4,5,6,7,8,12$. These walls yield the full SBLD of $\EFF(\PP^{2[n]})$.

\medskip\noindent
Second, we verify that the destabilizing objects for $\PP^{2[n]}$ all occur as subcomplexes of the minimal free resolution of elements $Z\in \PP^{2[n]}$. More precisely, this section proves: 

\begin{prop}
For the examples in this section, every sheaf $U$ in the Hilbert scheme is $F$-admissible, where $F$ is a sheaf that destabilizes it in the sense of Bridgeland. 
\end{prop}

\medskip\noindent
In other words, Conjecture \ref{conj} holds for the examples of this section. 

\begin{remark}
A few remarks on the proposition.
\begin{enumerate}
    \item The previous proposition verifies that for these examples if $U$ is destabilized at $W_x$, the Bridgeland wall with center $x$, and is general in a component of the stable base locus of the divisor class $H+\frac{1}{2y}B-\epsilon H$ for all small enough $\epsilon >0$, then $x=\frac{3}{2}-y.$
    \item In fact, the stronger statement that 
    the Bridgeland destabilizing triangle $F \to U \to W \to \cdot$ makes $U$ into a $(F,W)$-cone (relaxing the notion of $(F,W)$-cone to allow $W$ to be a Bridgeland semi-stable two-term complex of line bundles) holds. 
    \item In many cases, several distinct Bridgeland destabilizing objects determine the same wall. In each case, we have primarily chosen one, but list each of the others. The proposition holds for all destabilizing objects $F$.
\end{enumerate}
\end{remark}

\begin{definition}
The \textit{syzygy locus}, denoted $\mathcal{G}_i$, is the closure of the locus of subschemes of length $n$ with fixed Betti table $G_i$.
\end{definition}

\medskip\noindent
The tables below run the stable base locus decomposition program from left to right, and then top to bottom: (1) Consider a general sheaf $U$ with Betti table $G$ which is $F$-admissible (column: Syzygies). (2) Solve the interpolation problem for $U$ (column: Interpolating bundle $V$). (3) Consider the stable base locus of the Brill-Noether divisor $D_V$ (column: $\mathbf{Bs}(D_V)$). We then continue the program on the next horizontal line with a new syzygy locus. 

\medskip\noindent
For example, let $U\in M(\xi)$ be a sheaf with Betti table $G$ and $F$-admissible. Then, we list the minimal free resolution of the bundle $V$ that solves the interpolation problem for $U$. The base locus $\mathbf{Bs}(D_V)$ is denoted by $\mathcal{G}_1$, if it is the syzygy locus of the Betti table $G_1$:

\begin{center}
\begin{tabular}{c|c|c} 
Syzygies & Interpolating bundle $V$ & $\mathbf{Bs}(D_V)$   \\  \hline
$G,F$ & $V$ & $\mathcal{G}_1$\\ 
\end{tabular}
\end{center}

\medskip\noindent
Thus, the column `Syzygies' records the wall-crossing induced by the syzygies.  Indeed, each horizontal line represents a wall in the SBLD. In the tables below, a dotted line is not a wall of the SBLD, but distinguishes loci with different Betti tables, or $F$-admissibility structures, within a chamber. 


\medskip\noindent
The class of the Brill-Noether divisor induced by the bundle $V$ is 
$$\mu(V)H-\tfrac{1}{2}B \quad \in \  \mathrm{Pic}(\PP^{2[n]})\otimes \mathbb{R} .$$
Hence we can deduce, from the column in the middle, the slope $\mu(D_V)$ of the ray spanned by $D_V$ so we include that (column: $\mu(D_V)$). If the locus of $F$-admissible sheaves $G,F$ has one of the geometric descriptions listed below, then we include it as well (column: Geometry).

\medskip\noindent
\textbf{Notation:}
\begin{itemize}
    \item Let ``Gaeta general" be the locus of sheaves which have a general Gaeta resolution.    
    \item Let $L_k(n)$ be the locus of schemes of length $n$ with a linear subscheme of length at least $k$.
    \item Let $Q_k(n)$ be the locus of schemes of length $n$ with a subscheme of length at least $k$ on a conic.
    \item Let $C_k(n)$ be the locus of schemes of length $n$ with a subscheme of length at least $k$ on a cubic.
\end{itemize}

\medskip\noindent
We also use the notation $\mathcal{U}^C$ defined in Section \ref{sec: eff}, and write $\mathcal{I}_k$ for ideal sheaves of $k$ points appearing as destabilizing objects.

\medskip\noindent
We verify that the Bridgeland destabilizing objects occur following \cite{LZ,ABCH}. We color in {\color{gray}gray} the sub complex of the resolution of $\mathcal{I}_Z$ that makes it $F$-admissible and exhibits the resolution of a Bridgeland destabilizing object. 
Note that we take the convention that if a power of an object destabilizes $U$, we consider the highest such power.



\subsection*{Geometry of syzygies on $\PP^{2[3]}$} There are only two Betti tables of ideal sheaves in $\PP^{2[3]}$:
\begin{align*}
    G: \quad  & \OOT(-3)^2\xrightarrow{\hspace{7mm}}{\color{gray}\OOT(-2)^3}  \\[1.5mm]
    G_1: \quad  & \OOT(-4)\xrightarrow{\hspace{7mm}}\OOT(-3)\oplus{\color{gray}\OOT(-1)}
\end{align*}

\medskip\noindent The syzygies that determine the nontrivial (primary) extremal wall of $\EFF(\PP^{2[3]})$ are the following. The controlling exceptional bundle of $M(\xi)=\PP^{2[3]}$ is $\mathcal{O}_{\PP^2}(1)$, and its character $e_{1}$ satisfies $\chi(e_{1},\xi)=0$. Then, the primary extremal divisor in $\EFF(\PP^{2[3]})$ is described by Corollary \ref{cor: EffRigid}, Part (b). Indeed, the Gaeta resolution of any $Z\in\PP^{2[3]}$, with Betti table $G$ above, is $\mathcal{O}_{\PP^2}(-2)^3$-admissible and can be seen to satisfy $h^i(\mathcal{O}_{\PP^2}(1)\otimes \mathcal{I}_Z) = 0$ for $i=0,1,2$ by tensoring $G$ by $\mathcal{O}_{\PP^2}(1)$. The ideal sheaves that fail to be $\mathcal{O}_{\PP^2}(-2)^3$-admissible, i.e., the locus $\mathcal{U}^c$, forms a divisor which spans the primary extremal ray of $\EFF(\PP^{2[3]})$. 
In \cite{Hui14}, $\mathcal{U}^c$ is described as the Brill-Noether divisor of the bundle $\mathcal{O}_{\PP^2}(1)$, which is exactly the locus $\mathcal{G}_1 = L_3(3)$.

\medskip\noindent
The syzygies that determine the next wall of $\EFF(\PP^{2[3]})$, which is the extremal wall of $\Mov(\PP^{2[3]})$, are the following. A generic $Z\in \mathcal{U}^c$ is $\mathcal{O}_{\PP^2}(-1)$-admissible. This means its minimal free resolution yields a triangle $ \mathcal{O}_{\PP^2}(-1)\to \mathcal{I}_Z\to W\to \cdot$. 
By tensoring the minimal free resolution of $\mathcal{I}_Z$ with the following bundle $M_1$,
$$0\xrightarrow{\hspace{7mm}} \mathcal{O}_{\PP^2}(-1)^4\xrightarrow{\hspace{7mm}} \mathcal{O}_{\PP^2}^6 \xrightarrow{\hspace{7mm}} M_1\xrightarrow{\hspace{7mm}} 0,$$  
we have that $h^i(M_1\otimes \mathcal{I}_Z)=0$ for $i=0,1,2$. It turns out that $M_1$ solves the interpolation problem for the general $Z\in \mathcal{U}^c$ and its Brill-Noether divisor $D_{M_1}$ fails to contain $\mathcal{U}^c=L_3(3)$ in its base locus. Consequently, the class of $D_{M_1}$ is in $\Mov(\PP^{2[3]})$, and moreover it is extremal as we can exhibit a curve class $\beta\subset \PP^{2[3]}$ that sweeps out $L_3(3)$ such that $\beta . D_{M_1}=0$. Such a curve is induced by a pencil of 3 points on a fixed line. Therefore, $D_{M_1}$ spans the primary extremal wall of $\MOV(\PP^{2[3]})$. Also, this shows that $L_3(3)$ is a component of the SBLD. The class $D_{M_1}$ is nef and has empty base locus; hence our program terminates here. 
Summarizing:

\begin{center}
\begin{tabular}{cc|c|cc} 
Geometry  &Syzygies& Interpolating bundle $V$& $\mathbf{Bs}(D_V)$ & $\mu(D_V)$\\[1mm]  \hline
Gaeta general & $G,\ \OOT(-2)$ & $\OOT(1)$ & $\mathcal{U}^c=L_3(3)$ & 1 \\[2mm] \hline
$L_3(3)$ & $G_1,\  \OOT(-1)$ & $M_1$  & $\emptyset$ &  2 \\[2mm]
\end{tabular}
\end{center}

\medskip\noindent
Citing \cite[\S 9]{ABCH}, below are the Bridgeland destabilizing objects needed to compute \textit{all} the walls for $\PP^{2[3]}$. They all occur in the minimal free resolutions. 

\begin{enumerate}
\item A rank 3 wall $W_{-5/2}$ corresponding to the destabilizing object $\mathcal{O}_{\PP^2}(-2)^3$.\\ This wall is also given by $\mathcal{I}_1(-1)$ and $\mathcal{T}_{\PP^2}(-3)$.
\item A rank 1 wall $W_{-7/2}$ corresponding to the destabilizing object $\mathcal{O}_{\PP^2}(-1)$.
\end{enumerate}

\medskip
\subsection*{Geometry of syzygies on $\PP^{2[4]}$}
There are 3 Betti tables of ideal sheaves in $\PP^{2[4]}$:
\begin{align*}
    G: \quad  & \OOT(-4)\xrightarrow{\hspace{7mm}}{\color{gray}\OOT(-2)^2}  \\[1.5mm]
    G_1: \quad  & \OOT(-4)\oplus{\color{gray}\OOT(-3)}\xrightarrow{\hspace{7mm}}\OOT(-3)\oplus{\color{gray}\OOT(-2)^2}  \\[1.5mm]
    G_2: \quad  & \OOT(-5)\xrightarrow{\hspace{7mm}}\OOT(-4)\oplus{\color{gray}\OOT(-1)}
\end{align*}

\medskip\noindent
The syzygies that determine the primary extremal wall of $\EFF(\PP^{2[4]})$ are entirely analogous to those determining the primary extremal wall of $\EFF(\PP^{2[3]})$ after replacing $\mathcal{O}_{\PP^2}(-2)^3$ by $\mathcal{O}_{\PP^2}(-2)^2$, $\mathcal{O}_{\PP^2}(1)$ by $\mathcal{T}_{\PP^2}$, and $L_3(3)$ by $L_3(4)$.

\medskip\noindent
The syzygies that determine the extremal wall of $\Mov(\PP^{2[4]})$ are also analogous to those determining the extremal wall of $\Mov(\PP^{2[3]})$ after replacing $\mathcal{O}_{\PP^2}(-1)$ by $\mathcal{I}_1(-1)$, $M_1$ with the bundle 
$$0\xrightarrow{\hspace{7mm}} \mathcal{O}_{\PP^2}(-1)^2\xrightarrow{\hspace{7mm}} \mathcal{O}_{\PP^2}^2\oplus \OO(1)^2\xrightarrow{\hspace{7mm}} M_1\xrightarrow{\hspace{7mm}} 0,$$ and the empty base locus with $L_4(4)$.

\medskip\noindent
Finally, the syzygies that determine the second wall of $\Mov(\PP^{2[4]})$ are entirely analogous to those determining the extremal wall of $\Mov(\PP^{2[3]})$ after replacing $M_1$ with the bundle 
$$0\xrightarrow{\hspace{7mm}} \mathcal{O}_{\PP^2}(-1)^6\xrightarrow{\hspace{7mm}} \mathcal{O}_{\PP^2}^8\xrightarrow{\hspace{7mm}} M_2\xrightarrow{\hspace{7mm}} 0.$$
Our program terminates here. Summarizing:

\begin{center}
\begin{tabular}{cc|c|cc} 
Geometry  &Syzygies& Interpolating bundle $V$& $\mathbf{Bs}(D_V)$ & $\mu(D_V)$\\[1mm]  \hline
Gaeta general & $G,\ \OOT(-2)$ & $\mathcal{T}_{\PP^2}$ & $\mathcal{U}^c=L_3(4)$ & 3/2 \\[2mm] \hline
$L_3(4)$ & $G_1,\  \mathcal{I}_1(-1)$ & $M_1$  & $L_4(4)$ &  2 \\[2mm]\hline
$L_4(4)$ & $G_2,\  \OOT(-1)$ & $M_2$  & $\emptyset$ &  3 \\[2mm]
\end{tabular}
\end{center}

\medskip\noindent
Citing \cite[\S 9]{ABCH}, below are the Bridgeland destabilizing objects needed to compute \textit{all} the walls for $\PP^{2[4]}$. They all occur in the minimal free resolutions.

\begin{enumerate}
\item A rank 2 wall $W_{-3}$ corresponding to the destabilizing object $\mathcal{O}_{\PP^2}(-2)^2$.
\item A rank 1 wall $W_{-7/2}$ corresponding to the destabilizing object $\mathcal{I}_1(-1)$.
\item A rank 1 wall $W_{-9/2}$ corresponding to the destabilizing object $\mathcal{O}_{\PP^2}(-1)$.
\end{enumerate}

\subsection*{Geometry of syzygies on $\PP^{2[5]}$}
There are 3 Betti tables of ideal sheaves in $\PP^{2[5]}$:
\begin{align*}
    G: \quad  & \OOT(-4)^2\xrightarrow{\hspace{7mm}}\OOT(-3)^2\oplus{\color{gray}\OOT(-2)}  \\[1.5mm]
    G_1: \quad  & \OOT(-5)\oplus{\color{gray}\OOT(-3)}\xrightarrow{\hspace{7mm}}\OOT(-4)\oplus{\color{gray}\OOT(-2)^2}  \\[1.5mm]
    G_2: \quad  & \OOT(-6)\xrightarrow{\hspace{7mm}}\OOT(-5)\oplus{\color{gray}\OOT(-1)}
\end{align*}

\medskip\noindent
The syzygies associated to the extremal wall of $\EFF(\PP^{2[5]})$ are the following.  The controlling exceptional bundle of
$M(\xi)=\PP^{2[5]}$ is $\mathcal{O}_{\PP^2}(2)$, and its character $e_{2}$ satisfies $\chi(e_2,\xi)>0$. It follows that the generalized Gaeta resolution coincides with the Gaeta resolution. Hence, any $Z\in \mathcal{U}=\mathcal{G}$ is $\mathcal{O}_{\PP^2}(-2)$-admissible so its minimal free resolution yields the triangle $\mathcal{O}_{\PP^2}(-2)\to \mathcal{I}_Z\to W\to\cdot  .$ 
By tensoring the minimal free resolution of $\mathcal{I}_Z$ with the following bundle $M$,
$$0\xrightarrow{\hspace{7mm}} \mathcal{O}_{\PP^2}^2\xrightarrow{\hspace{7mm}} \mathcal{O}_{\PP^2}(1)^4\xrightarrow{\hspace{7mm}} M \xrightarrow{\hspace{7mm}} 0,$$
we have that $h^i(M\otimes \mathcal{I}_Z)=0$ for $i=0,1,2$. 
It turns out that $M$ solves the the interpolation problem for the general $Z\in \PP^{2[5]}$ and its Brill-Noether divisor $D_{M}$ fails to have $\mathcal{U}^c=L_4(5)$ in its base locus. Consequently, the class of $D_{M}$ is in $\EFF(\PP^{2[4]})$, and moreover it is extremal as we can exhibit a curve class $\beta\subset \PP^{2[5]}$ that sweeps out $\PP^{2[5]}$ such that $\beta . D_{M}=0$. 
Such a curve is induced by a pencil of 4 points in a line, plus a general fixed point. 
Therefore, $D_{M}$ spans the primary extremal wall of $\EFF(\PP^{2[5]})$.

\medskip\noindent
The syzygies that determine the second wall of $\Mov(\PP^{2[5]})$ are entirely analogous to those determining the extremal wall of $\Mov(\PP^{2[4]})$ after replacing $M_1$ with the bundle 
$$0\xrightarrow{\hspace{7mm}} \mathcal{O}_{\PP^2}(-1)^4\xrightarrow{\hspace{7mm}} \mathcal{O}_{\PP^2}^4\oplus \OO(1)^2\xrightarrow{\hspace{7mm}} M_1\xrightarrow{\hspace{7mm}} 0$$ and $L_4(4)$ with $L_5(5)$.

\medskip\noindent
The syzygies that determine the third wall of $\Mov(\PP^{2[5]})$ are entirely analogous to those determining the second wall of $\Mov(\PP^{2[4]})$ after replacing $M_2$ with the bundle 
$$0\xrightarrow{\hspace{7mm}} \mathcal{O}_{\PP^2}(-1)^8\xrightarrow{\hspace{7mm}} \mathcal{O}_{\PP^2}^{10}\xrightarrow{\hspace{7mm}} M_2\xrightarrow{\hspace{7mm}} 0.$$
In particular, our program terminates here. Summarizing:

\begin{center}
\begin{tabular}{cc|c|cc} 
Geometry  &Syzygies& Interpolating bundle $V$& $\mathbf{Bs}(D_V)$ & $\mu(D_V)$\\[1mm]  \hline
Gaeta general & $G,\ \OOT(-2)$ & $M$ & $\mathcal{U}^c=L_4(5)$ & 2 \\[2mm] \hline
$L_4(5)$ & $G_1,\  \mathcal{I}_1(-1)$ & $M_1$  & $L_5(5)$ &  3 \\[2mm]
\hline
$L_5(5)$ & $G_2,\  \OOT(-1)$ & $M_2$  & $\emptyset$ &  4 \\[2mm]
\end{tabular}
\end{center}

\medskip\noindent
Citing \cite[\S 9]{ABCH}, below are the Bridgeland destabilizing objects needed to compute \textit{all} the walls for $\PP^{2[5]}$. They all occur in the minimal free resolutions.

\begin{enumerate}
\item A rank 1 wall $W_{\frac{-7}{2}}$ corresponding to the destabilizing object $\mathcal{O}_{\PP^2}(-2)$.
\\ This wall is also given by $\mathcal{I}_2(-1)$.
\item A rank 1 wall $W_{\frac{-9}{2}}$ corresponding to the destabilizing object $\mathcal{I}_1(-1)$.
\item A rank 1 wall $W_{\frac{-11}{2}}$ corresponding to the destabilizing object $\mathcal{O}_{\PP^2}(-1)$.
\end{enumerate}

\subsection*{Geometry of syzygies on $\PP^{2[6]}$}
There are 5 Betti tables of ideal sheaves in $\PP^{2[6]}$:
\begin{align*}
    G: \quad  & \OOT(-4)^3\xrightarrow{\hspace{7mm}}{\color{gray}\OOT(-3)^4}  \\[1.5mm]
    G_1: \quad  & \OOT(-5)\xrightarrow{\hspace{7mm}}\OOT(-3)\oplus{\color{gray}\OOT(-2)}  \\[1.5mm]
    G_2: \quad  & \OOT(-5)\oplus{\color{gray}\OOT(-4)}\xrightarrow{\hspace{7mm}}\OOT(-4)\oplus{\color{gray}\OOT(-3)\oplus\OOT(-2)}\\[1.5mm]
    G_3: \quad  & \OOT(-6)\oplus{\color{gray}\OOT(-3)}\xrightarrow{\hspace{7mm}}\OOT(-5)\oplus{\color{gray}\OOT(-2)^2}\\[1.5mm]
    G_4: \quad  & \OOT(-7)\xrightarrow{\hspace{7mm}}\OOT(-6)\oplus{\color{gray}\OOT(-1)}
\end{align*}
where $\mathcal{G}_1=Q_6(6)$, $\mathcal{G}_2=L_4(6)$, $\mathcal{G}_3=L_5(6)$ and $\mathcal{G}_4=L_6(6)$.

The syzygies determining the walls $\EFF(\PP^{2[6]})$ are entirely analogous to the walls of the previous cases using the sheaf listed under syzygies below (whose corresponding complex is colored gray in the complexes above), the interpolating bundle listed in the corresponding column below, and the base locus with the corresponding entry as well. In all further examples, we only remark on walls which exhibit new phenomena.

\medskip\noindent
We summarize the base locus decomposition program for $\PP^{2[6]}$ below.

\begin{center}
\begin{tabular}{cc|c|cc} 
Geometry  &Syzygies& Interpolating bundle $V$& $\mathbf{Bs}(D_V)$ & $\mu(D_V)$\\[1mm]  \hline
Gaeta general & $G,\ \OOT(-3)$ & $\OOT(2)$ & $\mathcal{U}^c=Q_6(6)$ & 2 \\[2mm] \hline
$Q_6(6)$ & $G_1,\  \OOT(-2)$ & $coker\left(\OOT^3\to\OOT(1)^5\right)$  & $L_4(6)$ &  5/2 \\[2mm]
\hline
$L_4(6)$ & $G_2,\  \mathcal{I}_2(-1)$ & $coker\left(\OOT(-1)^2\to\OOT(1)^4\right)$  & $L_5(6)$ &  3 \\[2mm]
\hline
$L_5(6)$ & $G_3,\  \mathcal{I}_1(-1)$ & $coker\left(\OOT(-1)^6\to\OOT^6\oplus\OOT(1)^2\right)$  & $L_6(6)$ &  4 \\[2mm]
\hline
$L_6(6)$ & $G_4,\  \OOT(-1)$ & $coker\left(\OOT(-1)^{10}\to\OOT^{12}\right)$  & $\emptyset$ &  5 \\[2mm]
\end{tabular}
\end{center}

\medskip\noindent
Citing \cite[\S 9]{ABCH}, below are the Bridgeland destabilizing objects needed to compute \textit{all} the walls for $\PP^{2[6]}$. They all occur in the minimal free resolutions.

\begin{enumerate}
\item A rank 4 wall $W_{\frac{-7}{2}}$ corresponding to the destabilizing object $\mathcal{O}_{\PP^2}(-3)^4$.
\\ This wall is also given by $\mathcal{I}_1(-2)$, $\mathcal{I}_3(-1)$, and $\mathcal{T}_{\PP^2}(-4)$. Note that $\mathcal{I}_1(-2)$ and $\mathcal{I}_3(-1)$ only destabilize some of the objects destabilized along the wall in contrast to all previous destabilizing objects.
\item A rank 1 wall $W_{-4}$ corresponding to the destabilizing object $\mathcal{O}_{\PP^2}(-2)$.
\item A rank 1 wall $W_{\frac{-9}{2}}$ corresponding to the destabilizing object $\mathcal{I}_2(-1)$.
\item A rank 1 wall $W_{\frac{-11}{2}}$ corresponding to the destabilizing object $\mathcal{I}_1(-1)$.
\item A rank 1 wall $W_{\frac{-13}{2}}$ corresponding to the destabilizing object $\mathcal{O}_{\PP^2}(-1)$.
\end{enumerate}

    \subsection*{Geometry of syzygies on $\PP^{2[7]}$}  \medskip\noindent
The list below contains all the Betti tables that occur for ideal sheaves in $\PP^{2[7]}$ ordered by the dimension of their syzygy locus. This example exhibits reducible base locus.
\begin{align*}
G: \quad  & \mathcal{O}_{\PP^2}(-5)\oplus {\color{gray}\mathcal{O}_{\PP^2}(-4)}\xrightarrow{\hspace{7mm}}{\color{gray}\mathcal{O}_{\PP^2}(-3)^3}  \\[1.5mm]
G_1: \quad  & \mathcal{O}_{\PP^2}(-5)\oplus{\color{gray}\mathcal{O}_{\PP^2}(-4)^2}\xrightarrow{\hspace{7mm}}\mathcal{O}_{\PP^2}(-4)\oplus{\color{gray}\mathcal{O}_{\PP^2}(-3)^3}  \\[1.5mm]
G_2: \quad  & \mathcal{O}_{\PP^2}(-5)^2\xrightarrow{\hspace{7mm}}\mathcal{O}_{\PP^2}(-4)^2\oplus{\color{gray}\mathcal{O}_{\PP^2}(-2)} \\[1.5mm]
G_3: \quad  &\mathcal{O}_{\PP^2}(-6)\oplus{\color{gray}\mathcal{O}_{\PP^2}(-4)}\xrightarrow{\hspace{7mm}}\mathcal{O}_{\PP^2}(-5)\oplus{\color{gray}\mathcal{O}_{\PP^2}(-3)\oplus(-2)}  \\[2mm]
G_4: \quad  & \mathcal{O}_{\PP^2}(-7)\oplus{\color{gray}\mathcal{O}_{\PP^2}(-3)}\xrightarrow{\hspace{7mm}}\mathcal{O}_{\PP^2}(-6)\oplus{\color{gray}\mathcal{O}_{\PP^2}(-2)^2}  \\[1.5mm]
G_5: \quad  & \mathcal{O}_{\PP^2}(-8)\xrightarrow{\hspace{7mm}}\mathcal{O}_{\PP^2}(-7)\oplus{\color{gray}\mathcal{O}_{\PP^2}(-1)}  \\[1.5mm]
\end{align*}

\noindent
We only discuss walls displaying new behavior. 
For $\PP^{2[7]}$, the only two such walls are the second and third walls of $\EFF(\PP^{2[7]})$. The second wall is the first instance where the Betti numbers of the generic element of the base locus do not change from one chamber to the next. The third wall is the first instance where the stable base locus of a chamber is reducible and each of the two irreducible components has a distinct destabilizing object. 
Let us elaborate. 

\medskip\noindent
The syzygies that determine the second wall of $\EFF(\PP^{2[7]})$ are the following. 
The base locus of the extremal effective chamber is $D_{E_{12/5}} = \mathcal{U}^C = Q_6(7)$.
We want to show that a generic $Z\in Q_6(7)$ is $\mathcal{I}_{p}(-2)$-admissible. 
To see this, first note that (similarly to the example in the Introduction), the general element of $Q_6(7)$ has Betti table $G$, i.e., has a  resolution 
\begin{align*}
0\xrightarrow{\hspace{7mm}} \OO(-4)\oplus\OO(-5)\overset{\left(
\begin{array}{cc}
l_1  & q_1 \\ 
 l_2 & q_2 \\
  l_3 & q_3 \\
\end{array}\right)}{\xrightarrow{\hspace{7mm}}} \OO(-3)^3\xrightarrow{\hspace{7mm}} \mathcal{I}_Z\xrightarrow{\hspace{7mm}} 0.
\end{align*}
In other words, any sheaf in $\mathcal{G}$ has three cubic generators, $c_1$, $c_2$, and $c_3$, one quartic relation, $q = \ell_1c_1+\ell_2c_2+\ell_3c_3$ where $\ell_i$ is degree one, and one quintic relation $Q = q_1c_1+q_2c_2+q_3c_3$ where $q_i$ is degree two.
A sheaf in $\mathcal{G}$ lies in $Q_6(7)$ exactly when, up to change of coordinates, one of the $\ell_i$ is zero.
With that zero, the minimal free resolution yields a triangle $ \mathcal{I}_1(-2)\to \mathcal{I}_Z\to W\to \cdot$.
From there, the rest of the discussion is analogous to previous cases with the interpolating bundle $M_1$ given below
$$0\xrightarrow{\hspace{7mm}} \mathcal{O}_{\PP^2}\xrightarrow{\hspace{7mm}} \mathcal{O}_{\PP^2}(1)\oplus \OO(2)^2\xrightarrow{\hspace{7mm}} M_1\xrightarrow{\hspace{7mm}} 0.$$  

\medskip\noindent
The syzygies associated to the third wall of $\EFF(\PP^{2[7]})$ are the following. As mentioned above, this wall exhibits a reducible stable base locus where both irreducible components are destabilized. Consider a general $Z\in \mathcal{G}_1$. The Betti tables of $\mathcal{I}_Z$ and $M_1$ imply that $h^0(M_1\otimes \mathcal{I}_Z)\neq 0$. Hence, $Z\in \textbf{Bs}(D_{M_1})$. Furthermore, if we consider the curve $\beta\subset \mathcal{G}_1=L_4(7)$ induced by a general pencil of 4 points on a fixed line, and 3 additional general fixed points, then $\beta \cdot D=0$ where the class $D$ is a multiple of $3H-\tfrac{1}{2}\Delta$. Since $\mu(D_{M_1})<3$, then $L_4(7)\subset \mathbf{Bs}(D_{M_1})$. 
Since $\mathcal{I}_Z$ is $\mathcal{I}_3(-1)$-admissible, then the following bundle is numerically orthogonal to $\mathcal{I}_Z$  $$0\xrightarrow{\hspace{7mm}} \OO^6\xrightarrow{\hspace{7mm}} \OO(1)^9\xrightarrow{\hspace{7mm}} M_2\xrightarrow{\hspace{7mm}} 0.$$
If follows from the minimal free resolutions of $\mathcal{I}_Z$ and $M_2$ that $h^0(M_2\otimes \mathcal{I}_Z)=0$ for $i=0,1,2$. The existence of the curve $\beta$ above implies that $M_2$ solves the interpolation problem for a generic $Z\in \mathcal{G}_1$. Then, $\mathcal{G}_1$ is not in the base locus of $D_{M_2}$. Similarly, a general pencil of 7 points on a fixed smooth conic induces a curve $\alpha\subset Q_7(7)$ with $\alpha\cdot D_{M_2}=0$. Hence $Q_7(7)\subset \mathbf{Bs}(D_{M_1})$. Note that $M_2$ also solves the interpolation problem for a generic element in $\mathcal{G}_2=Q_7(7)$. Consequently, $\mathcal{G}_1\cup \mathcal{G}_2$ both enter the base locus at the wall spanned by $D_{M_2}$.

\medskip\noindent
We summarize the base locus decomposition program for $\PP^{2[7]}$ below.

\begin{center}
\begin{tabular}{cc|c|cc} 
Geometry  &Syzygies& Interpolating bundle $V$& $\mathbf{Bs}(D_V)$ & $\mu(D_V)$\\[1mm]  \hline
Gaeta general & $G,\ \mathcal{T}_{\PP^2}(-4)$ & $E_{12/5}$ & $\mathcal{U}^c=Q_6(7)$ & 12/5 \\[2mm] \hline
$Q_6(7)$ & $G,\  \mathcal{I}_1(-2)$ & $M_1$  & $L_4(7)\cup Q_7(7)$ &  5/2 \\[2mm] \hline
$L_4(7) $ & $G_1, \mathcal{I}_3(-1)$ & $M_2$ & $L_5(7)$ &  3 \\[2mm] \hdashline
$Q_7(7)$ & $G_2,\  \mathcal{O}_{\PP^2}(-2)$ & $M_2$ & $L_5(7)$ &  3 \\[2mm] \hline
$L_5(7)$ & $G_3,\  \mathcal{I}_2(-1)$ & $coker\left(\OOT(-1)^4\to\OOT^2\oplus\OOT(1)^4\right)$ & $L_6(7)$ &  4 \\[2mm] \hline 
$L_6(7)$ & $G_4,\  \mathcal{I}_1(-1)$ & $coker\left(\OOT(-1)^8\to\OOT^8\oplus\OOT(1)^2\right)$ & $L_7(7)$ & 5  \\[2mm] \hline 
$L_7(7)$ & $G_5,\  \mathcal{O}_{\PP^2}(-1)$ &  $coker\left(\OOT(-1)^{12}\to\OOT^{14}\right)$  & $\emptyset$ & 6  \\[2mm] 
\end{tabular}
\end{center}

\medskip\noindent
\medskip\noindent
Citing \cite[\S 9]{ABCH}, below are the Bridgeland destabilizing objects needed to compute \textit{all} the walls for $\PP^{2[7]}$. They all occur in the minimal free resolutions.

\begin{enumerate}
\item A rank 2 wall $W_{\frac{-39}{10}}$ corresponding to the destabilizing object $\mathcal{T}_{\PP^2}(-4)$.
\item A rank 1 wall $W_{-4}$ corresponding to the destabilizing object $\mathcal{I}_1(-2)$.
\item A rank 1 wall $W_{\frac{-9}{2}}$ corresponding to the destabilizing objects $\mathcal{I}_3(-1)$ and $\mathcal{O}_{\PP^2}(-2)$. 
\item A rank 1 wall $W_{\frac{-11}{2}}$ corresponding to the destabilizing object $\mathcal{I}_2(-1)$.
\item A rank 1 wall $W_{\frac{-13}{2}}$ corresponding to the destabilizing object $\mathcal{I}_1(-1)$.
\item A rank 1 wall $W_{\frac{-15}{2}}$ corresponding to the destabilizing object $\mathcal{O}_{\PP^2}(-1)$.
\end{enumerate}

\subsection*{Geometry of syzygies on $\PP^{2[8]}$}
\medskip\noindent
The list below contains all the Betti tables that occur for ideal sheaves in $\PP^{2[8]}$ ordered by the dimension of their syzygy locus.  

\begin{align*}
G: \quad  &   \mathcal{O}_{\PP^2}(-5)^2\xrightarrow{\hspace{7mm}}\mathcal{O}_{\PP^2}(-4)\oplus {\color{gray}\mathcal{O}_{\PP^2}(-3)^2}  \\[1.5mm]
G_1: \quad  & \mathcal{O}_{\PP^2}(-5)^2\oplus{\color{gray}\mathcal{O}_{\PP^2}(-4)}\xrightarrow{\hspace{7mm}}\mathcal{O}_{\PP^2}(-4)^2\oplus{\color{gray}{\mathcal{O}_{\PP^2}(-3)^2}}  \\[1.5mm]
G_2: \quad  &  \mathcal{O}_{\PP^2}(-6)\oplus {\color{gray}\mathcal{O}_{\PP^2}(-4)^2}\xrightarrow{\hspace{7mm}}\mathcal{O}_{\PP^2}(-5)\oplus {\color{gray}\mathcal{O}_{\PP^2}(-3)^3}\\[2mm]
G_3: \quad  & \mathcal{O}_{\PP^2}(-6)\xrightarrow{\hspace{7mm}}\mathcal{O}_{\PP^2}(-4)\oplus{\color{gray}\mathcal{O}_{\PP^2}(-2) } \\[2mm]
G_4: \quad  &\mathcal{O}_{\PP^2}(-6)\oplus{\color{gray}\mathcal{O}_{\PP^2}(-5)}\xrightarrow{\hspace{7mm}}\mathcal{O}_{\PP^2}(-5)\oplus{\color{gray}\mathcal{O}_{\PP^2}(-4)\oplus\mathcal{O}_{\PP^2}(-2)}  \\[1.5mm]
G_5: \quad  & \mathcal{O}_{\PP^2}(-7)\oplus{\color{gray}\mathcal{O}_{\PP^2}(-4)}\xrightarrow{\hspace{7mm}}\mathcal{O}_{\PP^2}(-6)\oplus{\color{gray}\mathcal{O}_{\PP^2}(-3)\oplus\mathcal{O}_{\PP^2}(-2)}  \\[1.5mm]
G_6: \quad  & \mathcal{O}_{\PP^2}(-8)\oplus{\color{gray}\mathcal{O}_{\PP^2}(-3)}\xrightarrow{\hspace{7mm}}\mathcal{O}_{\PP^2}(-7)\oplus{\color{gray}\mathcal{O}_{\PP^2}(-2)^2}  \\[1.5mm]
G_7: \quad  & \mathcal{O}_{\PP^2}(-9)\xrightarrow{\hspace{7mm}}\mathcal{O}_{\PP^2}(-8)\oplus{\color{gray}\mathcal{O}_{\PP^2}(-1)}  \\[1.5mm]
\end{align*}

\noindent
The syzygies that determine the extremal divisor of $\EFF(\PP^{2[8]})$ are the following. The controlling exceptional bundle of
$M(\xi)=\PP^{2[8]}$ is $\OO(3)$  and we have that $\chi(\OO(3),\xi)>0$. By Theorem \ref{MAINdetailed}, Part (a), the Gaeta resolution of a general $Z\in \PP^{2[8]}$ is $\mathcal{O}_{\PP^2}(-3)^2$-admissible and gives rise to the triangle $ \mathcal{O}_{\PP^2}(-3)^3\to \mathcal{I}_Z\to W\to \cdot$. This implies that the following bundle is numerically orthogonal to $\mathcal{I}_Z$ 
$$0\xrightarrow{\hspace{7mm}} \OO(1)^2\xrightarrow{\hspace{7mm}} \mathcal{O}_{\PP^2}(2)^5\xrightarrow{\hspace{7mm}} M\xrightarrow{\hspace{7mm}} 0.$$

\medskip\noindent
The Brill-Noether divisor $D_M$ generates the extremal ray of $\EFF(\PP^{2[8]})$. Indeed, the extremality of $D_M$ follows from considering a moving curve $\beta\subset \PP^{2[8]}$ induced by a general pencil of 8 points on a fixed smooth cubic and observing that $\beta \cdot D_M=0$. On the other hand, $D_M$ is effective as $h^i(M\otimes \mathcal{I}_Z)=0$ for $i=0,1,2$, which follows from the resolutions of $M$ and $\mathcal{I}_Z$.

\medskip\noindent
The base locus of $D_M$ consists of two components, $L_4(8)$ and $\mathcal{G}_1$, both of which are involved in the next wall. We include the next wall as the generic element of both irreducible components are Bridgeland destabilized but, in contrast to the case for 7 points, the base locus in the following chamber is still reducible.

\medskip\noindent
The syzygies associated to the second wall of $\EFF(\PP^{2[8]})$ are the following. 
We first deal with the syzygies coming from $L_4(8)$ (as we will see below).
The generic $Z\in \PP^{2[8]}$ has 2 generators $c,c'$ of degree 3, and a generator $q$ of degree 4. The Betti table $G$ tells us there are two syzygies in degree 5 of the following form: 
$$rc+r'c'=lq,$$
for some polynomials $r,r'$ of degree 2 and $l$ of degree 1. 
Now, consider $R \in \PP^{2[8]}$ with Betti table $G$, and whose generators of degree 3, denoted $c,c'$, have a special syzygy in degree 5 as follows: there exist quadratic polynomials $y,y'$ such that 
$$yc+y'c'=0.$$ 
Thus, there is  a triangle $\mathcal{I}_4(-1)\to \mathcal{I}_R\to W\to $, where 
$\OO(-5)\hookrightarrow \OO(-3)^2\twoheadrightarrow \mathcal{I}_4(-1)$. Observe that this triangle implies that $h^0(M\otimes\mathcal{I}_R)=1$, hence $\mathcal{I}_4$-admissible sheaves, with Betti table $G$, are in the base locus of $D_M$; these sheaves form the locus $L_4(8)$. Also, this triangle implies that the following bundle is numerically orthogonal to $\mathcal{I}_R$:
$$0\xrightarrow{\hspace{7mm}} \mathcal{O}_{\PP^2}^2\xrightarrow{\hspace{7mm}} \mathcal{O}_{\PP^2}(1)^2\oplus \OO(2)^2\xrightarrow{\hspace{7mm}} M_1\xrightarrow{\hspace{7mm}} 0.$$

\medskip\noindent
Then, $h^i(M_1\otimes \mathcal{I}_R)=0$, with $i=0,2$ follows by tensoring the resolution of $\mathcal{I}_R$ with $M_1$. It turns out that $M_1$ solves the the interpolation problem for $\mathcal{I}_R$ and its Brill-Noether divisor $D_{M_1}$ does not have $L_4(8)$ in is base locus. In order to show that $D_{M_1}$ spans a wall we must exhibit a curve class that sweeps out $L_4(8)$ whose intersection with $D_{M_1}$ is zero. Such a curve is induced by a pencil of 4 points on a fixed line (and 4 general fixed points). Therefore, $D_{M_1}$ spans the wall of $\EFF(\PP^{2[8]})$ where any $\mathcal{I}_R\in L_4(8)$ enters the base locus.

\medskip\noindent
We now turn to the syzygies coming from $\mathcal{G}_1$ (as we will see below). 
Indeed, consider $Z\in \mathcal{G}_1$ and note $h^1(M\otimes \mathcal{I}_Z)\neq 0$. Thus  $Z$ is in the base locus of $D_M$. Further, we have that $\mathcal{I}_Z$ is $\mathcal{I}_1(-2)$-admissible, where $\OO(-4)\hookrightarrow \OO(-3)^2\twoheadrightarrow \mathcal{I}_p(-2)$. Hence, $\mathcal{I}_Z$ fits into the following triangle $\mathcal{I}_p(-2)\to \mathcal{I}_Z\to W\to $, which implies that $M_1$ is numerically orthogonal to $\mathcal{I}_Z$. It turns out that $M_1$ solves the interpolation problem for $Z$. Also, the curve $\beta\subset \mathcal{G}_1=Q_7(8)$ induced by the general pencil of 7 points on a fixed smooth conic (and an 8th general fixed point) implies that $\mathcal{G}_1\subset \mathbf{Bs}(D_M)$. Hence, $L_4(8)\cup \mathcal{G}_1$ both enter the base locus at the wall spanned by $D_{M_1}$.

\medskip\noindent
Note that the syzygy locus $\mathcal{G}_4$ is in the closure of $\mathcal{G}_3$ and $\mathcal{G}_2$ and the destabilizing objects coming from it is not listed below because it fails to be general in any stable base locus component. Geometrically, $\mathcal{G}_4$ is the locus of 8 points lying on a reducible conic such that 5 of the points lie on one line and 3 lie on the other. Here we see that using syzygies allows us to refine the information gleaned from Bridgeland stability and the stable base locus decomposition.

\medskip\noindent
The remaining walls are analogous to previous discussed examples, so we summarize the base locus decomposition program for $\PP^{2[8]}$ below.

\begin{center}
\begin{tabular}{cc|c|cc} 
Geometry  &Syzygies& Interpolating bundle $V$& $\mathbf{Bs}(D_V)$ & $\mu(D_V)$\\[1mm]  \hline
Gaeta general & $G,\ \mathcal{O}_{\PP^2}(-3)^3$ & $M$ & $L_4(8) \cup Q_7(8)$ & 8/3 \\[2mm] \hline
$L_4(8)$ & $G,\  \mathcal{I}_4(-1)$ & $M_1$  & $L_5(8)\cup Q_8(8)$ &  3 \\[2mm] \hdashline
$Q_7(8) $ & $G_1, \mathcal{I}_1(-2)$ & $M_1$ & $L_5(8)\cup Q_8(8)$ &  3 \\[2mm] \hline
$Q_8(8)$ & $G_3,\  \mathcal{O}_{\PP^2}(-2)$ & $coker\left(\OOT^{10}\to\OOT(1)^{14}\right)$ & $L_5(8)$ & 7/2 \\[2mm] \hline
$L_5(8)$ & $G_2, \mathcal{I}_3(-1)$ & $coker\left(\OOT(-1)\to\OOT\oplus\OOT(1)^3\right)$ & $L_6(8)$ &  4 \\[2mm] \hline
$L_6(8)$ & $G_5, \mathcal{I}_2(-1)$ & $coker\left(\OOT(-1)^6\to\OOT^4\oplus\OOT(1)^4\right)$ & $L_7(8)$ &  5 \\[2mm] \hline
$L_7(8)$ & $G_6, \mathcal{I}_1(-1)$ & $coker\left(\OOT(-1)^{10}\to\OOT^{10}\oplus\OOT(1)^2\right)$ & $L_8(8)$ &  6 \\[2mm] \hline
$L_8(8)$ & $G_7,\  \mathcal{O}_{\PP^2}(-1)$ & $coker\left(\OOT(-1)^{14}\to\OOT^{16}\right)$ & $\emptyset$ &  7 \\[2mm] 
\end{tabular}
\end{center}

\medskip\noindent
Citing \cite[\S 9]{ABCH}, below are the Bridgeland destabilizing objects needed to compute \textit{all} the walls for $\PP^{2[8]}$. They all occur in the minimal free resolutions.

\begin{enumerate}
\item A rank 2 wall $W_{\frac{-25}{6}}$ corresponding to the destabilizing object $\mathcal{O}_{\PP^2}(-3)^2$.
\item A rank 1 wall $W_{\frac{-9}{2}}$ corresponding to the destabilizing objects $\mathcal{I}_4(-1)$ and $\mathcal{I}_1(-2)$. 
\item A rank 1 wall $W_{-5}$ corresponding to the destabilizing object $\mathcal{O}_{\PP^2}(-2)$.
\item A rank 1 wall $W_{\frac{-11}{2}}$ corresponding to the destabilizing object $\mathcal{I}_3(-1)$.
\item A rank 1 wall $W_{\frac{-13}{2}}$ corresponding to the destabilizing object $\mathcal{I}_2(-1)$.
\item A rank 1 wall $W_{\frac{-15}{2}}$ corresponding to the destabilizing object $\mathcal{I}_1(-1)$.
\item A rank 1 wall $W_{\frac{-17}{2}}$ corresponding to the destabilizing object $\mathcal{O}_{\PP^2}(-1)$.
\end{enumerate}

\subsection*{Geometry of syzygies on $\PP^{2[12]}$}  \medskip\noindent
The list below contains the Betti tables of the generic element of each irreducible stable base locus component in $\PP^{2[12]}$. The excluded Betti tables do not form components of the stable base locus. 
When a given Betti table has multiple relevant destabilizing objects, only one is colored. 

\begin{align*}
G\hspace{4pt} : \quad  &   \OOT(-6)^2\xrightarrow{\hspace{7mm}} {\color{gray}\OOT(-4)^3} \\
G_{1}: \quad  &   \OOT(-6)^2\oplus {\color{gray}\OOT(-5)} \xrightarrow{\hspace{7mm}} \OOT(-5) \oplus {\color{gray}\OOT(-4)^3} \\
G_{2}: \quad  &  \OOT(-6)^3 \xrightarrow{\hspace{7mm}}  \OOT(-5)^3 \oplus {\color{gray}\OOT(-3)} \\
G_{3}: \quad  &  \OOT(-6)^2\oplus {\color{gray}\OOT(-5)^2} \xrightarrow{\hspace{7mm}} \OOT(-5)^2 \oplus {\color{gray}\OOT(-4)^3}\\
G_{4}: \quad  &  \OOT(-7) \oplus {\color{gray}\OOT(-5)}\xrightarrow{\hspace{7mm}} \OOT(-5) \oplus {\color{gray}\OOT(-4)\oplus \OOT(-3)} \\
G_{5}: \quad  &  \OOT(-7)\oplus {\color{gray}\OOT(-5)^3} \xrightarrow{\hspace{7mm}} \OOT(-6) \oplus {\color{gray}\OOT(-4)^4}\\
G_{6}: \quad  &  \OOT(-7)^2 \oplus {\color{gray}\OOT(-4)} \xrightarrow{\hspace{7mm}}  \OOT(-6)^2 \oplus {\color{gray}\OOT(-3)^2} \\
G_{7}: \quad  &  \OOT(-8) \xrightarrow{\hspace{7mm}} \OOT(-6) \oplus {\color{gray}\OOT(-2)} \\
G_{8}: \quad  &  \OOT(-8)\oplus{\color{gray}\OOT(-5)^2}\xrightarrow{\hspace{7mm}}\OOT(-7)\oplus{\color{gray}\OOT(-4)^2\oplus\OOT(-3)}\\
G_{9}: \quad  &  \mathcal{O}_{\PP^2}(-9)\oplus {\color{gray}\mathcal{O}_{\PP^2}(-5)}\xrightarrow{\hspace{7mm}}\mathcal{O}_{\PP^2}(-8)\oplus{\color{gray}\OO_{\PP^2}(-3)^2}\\
G_{10}: \quad  &  \mathcal{O}_{\PP^2}(-10)\oplus {\color{gray}\mathcal{O}_{\PP^2}(-4)^2}\xrightarrow{\hspace{7mm}}\mathcal{O}_{\PP^2}(-9)\oplus{\color{gray}\OO_{\PP^2}(-3)^3}\\
G_{11}: \quad  &  \mathcal{O}_{\PP^2}(-11)\oplus {\color{gray}\mathcal{O}_{\PP^2}(-4)}\xrightarrow{\hspace{7mm}}\mathcal{O}_{\PP^2}(-10)\oplus{\color{gray}\OO_{\PP^2}(-3)\oplus\OO_{\PP^2}(-2)}\\
G_{12}: \quad  &  \mathcal{O}_{\PP^2}(-12)\oplus {\color{gray}\mathcal{O}_{\PP^2}(-3)}\xrightarrow{\hspace{7mm}}\mathcal{O}_{\PP^2}(-11)\oplus{\color{gray}\OO_{\PP^2}(-2)^2}\\
G_{13}: \quad  &   \mathcal{O}_{\PP^2}(-13)\xrightarrow{\hspace{7mm}}\mathcal{O}_{\PP^2}(-12)\oplus{\color{gray}\OO_{\PP^2}(-1)}
\end{align*}

\medskip\noindent
The novel aspect of this example, in addition to being a newly worked out example, is the behavior of sheaves with Betti table $G_1$.
Sheaves with that Betti table have resolution 
\begin{align*}
0\xrightarrow{\hspace{7mm}} \OO(-5) \oplus\OO(-6)^2\overset{\left(
\begin{array}{ccc}
0   & l_4 & l_5\\ 
l_1 & q_1 & q_4\\
l_2 & q_2 & q_5 \\
l_3 & q_3 & q_6 \\
\end{array}\right)}{\xrightarrow{\hspace{7mm}}} \OO(-5) \oplus \OO(-4)^3\xrightarrow{\hspace{7mm}} \mathcal{I}_Z\xrightarrow{\hspace{7mm}} 0.
\end{align*}
The general such sheaf is $\mathcal{T}_{\PP^2}(-5)$-admissible, is destabilized by the same sheaf, and is the generic sheaf in $\mathcal{G}_1$.
If $l_1$, $l_2$, and $l_3$ are linearly dependent, then the sheaf is $\mathcal{I}_{1}(-3)$-admissible,  is destabilized by the same sheaf, and is the generic sheaf in $C_{11}(12)$.
If $l_4$ and $l_5$ are linearly dependent, then the sheaf is $\mathcal{I}_{7}(-1)$-admissible,  is destabilized by the same sheaf, and is the generic sheaf in $L_5(12)$.
Note, two linear polynomials being linearly dependent is more conditions than three, and we see this confirmed in the codimension of the corresponding stable base locus components.

\medskip\noindent
We omit the explanation of all further walls, but list them in the table below. 

\begin{center}
\scalebox{.9}{
\begin{tabular}{cc|c|cc}
Geometry&Syzygies & Interpolating bundle $V$& $\mathbf{Bs}(D_V)$ & $\mu(D_V)$\\[1mm]  \hline
Gaeta general & $G, \mathcal{O}_{\PP^2}(-4)^3$ & $\mathcal{T}_{\PP^2}(2)$ & $\mathcal{G}_1$ & 7/2 \\[2mm] \hline
 & $G_1, \mathcal{T}_{\PP^2}(-5)$ & $coker\left(\OOT(1)^2\to\OOT(3)^9\right)$ & $C_{11}(12)$ & 25/7  \\[2mm] \hline 
$C_{11}(12)$&$G_{1}, \mathcal{I}_1(-3)$ & $coker\left(\OOT(1)^2\xrightarrow{\hspace{3mm}}\OOT(2)^2\oplus\OOT(3)^3\right)$  & $(L_{5}\cup Q_9\cup C_{12})(12)$ &11/3    \\[2mm] \hline
$C_{12}(12)$ & $G_{2}, \OOT(-3)$ & $coker\left(\OOT(1)^4\xrightarrow{\hspace{3mm}}\OOT(2)^6\right)$ & $Q_{10}(12)$& 4 \\[2mm]\hdashline
$Q_9(12)$ &$G_{3},\mathcal{I}_3(-2)$  & $coker\left(\OOT(1)^4\xrightarrow{\hspace{3mm}}\OOT(2)^6\right)$ & $Q_{10}(12)$ & 4 \\[2mm] \hdashline
$L_5(12)$ & $G_{1},\mathcal{I}_7(-1)$ &$coker\left(\OOT(1)^4\xrightarrow{\hspace{3mm}}\OOT(2)^6\right)$ & $Q_{10}(12)$ & 4\\[2mm]  \hline
$Q_{10}(12)$ &$G_{4}, \mathcal{I}_2(-2)$ &  $coker\left(\OOT^3\xrightarrow{\hspace{3mm}}\OOT(1)\oplus\OOT(2)^4\right)$ & $L_6(12)\cup Q_{11}(12)$ &9/2 \\[2mm] \hline
$L_6(12)$ & $G_{5}, \mathcal{I}_6(-1)$ & $coker\left(\OOT^6\xrightarrow{\hspace{3mm}}\OOT(1)^6\oplus\OOT(2)^2\right)$ & $Q_{12}(12)$ &  5\\[2mm] \hdashline 
$Q_{11}(12)$ & $G_{6},\mathcal{I}_1(-2)$ & $coker\left(\OOT^6\xrightarrow{\hspace{3mm}}\OOT(1)^6\oplus\OOT(2)^2\right)$ & $Q_{12}(12)$ &5  \\[2mm] \hline
$Q_{12}(12)$ & $G_7,\OOT(-2)$   & $coker\left(\OOT^9\xrightarrow{\hspace{3mm}}\OOT(1)^{11}\right)$& $L_7(12)$& 11/2 \\[1.5mm] \hline
$L_7(12)$ & $G_8,\mathcal{I}_5(-1)$ &  $coker\left(\OOT(-1)^2\oplus\OOT^6\xrightarrow{\hspace{3mm}}\OOT(1)^{10}\right)$&$L_8(12)$ & 6\\[1.5mm] \hline 
$L_8(12)$ & $G_9,\mathcal{I}_4(-1)$ & $coker\left(\OOT(-1)^6\xrightarrow{\hspace{3mm}}\OOT(1)^8\right)$  &$L_9(12)$ & 7\\[1.5mm] \hline 
$L_9(12)$ & $G_{10},\mathcal{I}_3(-1)$ & $coker\left(\OOT(-1)^{10}\xrightarrow{\hspace{3mm}}\OOT^6\oplus\OOT(1)^6\right)$ &$L_{10}(12)$ & 8\\[1.5mm] \hline 
$L_{10}(12)$ & $G_{11},\mathcal{I}_2(-1) $ & $coker\left(\OOT(-1)^{14}\xrightarrow{\hspace{3mm}}\OOT^{12}\oplus\OOT(1)^4\right)$ &$L_{11}(12)$ & 9\\[1.5mm] \hline 
$L_{11}(12)$ & $G_{12},\mathcal{I}_1(-1)$ & $coker\left(\OOT(-1)^{18}\xrightarrow{\hspace{3mm}}\OOT^{18}\oplus\OOT(1)^2\right)$ &$L_{12}(12)$ & 10\\[1.5mm] \hline 
$L_{12}(12)$ & $G_{13},\OOT(-1)$ & $coker\left(\OOT(-1)^{22}\xrightarrow{\hspace{3mm}}\OOT^{24}\right)$ &$\emptyset$ & 11\\[1.5mm] 
\end{tabular}}
\end{center}

\medskip\noindent
Following \cite{LZ}, below are the Bridgeland destabilizing objects needed to compute \textit{all} the walls for $\PP^{2[12]}$. They all occur in the minimal free resolutions.
Some syzygy loci are not listed below because they fail to be general in an irreducible component of the stable base locus.

\begin{enumerate}
\item A rank 3 wall $W_{-5}$ corresponding to the destabilizing object $\mathcal{O}_{\PP^2}(-3)^3$.\\
This wall is also given by general sheaves with log Chern character $(2,-3,\frac{3}{2})$ and for some sheaves by $\mathcal{I}_{4}(-2)$.
\item A rank 2 wall $W_{\frac{-71}{14}}$ corresponding to the destabilizing object $\mathcal{T}_{\PP^2}(-5)$
\item A rank 1 wall $W_{\frac{-31}{6}}$ corresponding to the destabilizing object $\mathcal{I}_{1}(-3)$.
\item A rank 1 wall $W_{\frac{-11}{2}}$ corresponding to the destabilizing objects  $\mathcal{O}_{\PP^2}(-3)$, $\mathcal{I}_{3}(-2)$, $\mathcal{I}_{7}(-1)$.
\item A rank 1 wall $W_{-6}$ corresponding to the destabilizing object $\mathcal{I}_{2}(-2)$.
\item A rank 1 wall $W_{\frac{-13}{2}}$ corresponding to the destabilizing objects $\mathcal{I}_{1}(-2)$ and $\mathcal{I}_{6}(-1)$.
\item A rank 1 wall $W_{-7}$ corresponding to the destabilizing object $\mathcal{O}_{\PP^2}(-2)$.
\item A rank 1 wall $W_{\frac{-15}{2}}$ corresponding to the destabilizing object $\mathcal{I}_{5}(-1)$.
\item A rank 1 wall $W_{\frac{-17}{2}}$ corresponding to the destabilizing object $\mathcal{I}_{4}(-1)$.
\item A rank 1 wall $W_{\frac{-19}{2}}$ corresponding to the destabilizing object $\mathcal{I}_{3}(-1)$.
\item A rank 1 wall $W_{\frac{-21}{2}}$ corresponding to the destabilizing object $\mathcal{I}_{2}(-1)$.
\item A rank 1 wall $W_{\frac{-23}{2}}$ corresponding to the destabilizing object $\mathcal{I}_{1}(-1)$.
\item A rank 1 wall $W_{\frac{-25}{2}}$ corresponding to the destabilizing object $\mathcal{O}_{\PP^2}(-1)$.
\end{enumerate}

\section*{Appendix A: Computations in Macaulay2}

\noindent 
This appendix contains an implementation, using Macaulay2, for proving Lemma \ref{InterpolationTan} for the cases $s\leq5$. The cases $s\leq 4$ can be treated uniformly so we only address $s=4$. In the case $s=5$, an optimization is possible using the Resolution \eqref{eq: esp res} in Lemma \ref{InterpolationTan}. 

\subsection*{Case $s=4$.}
\begin{verbatim}
PP2 = QQ[x, y, z];
--First we define the bundle M
B2 = flatten entries basis(2, PP2);
C4 = for i in 0..18 list (random(PP2^{2}, PP2^{0}))_0_0;
A = matrix for i in 0..18 list
	for j in 0..3 list(
		if j == 3 then C4_i
		else if j == floor(i/6) then B2_(i-6*j)
		else 0
	);
M = sheaf coker map(PP2^{7}^19, PP2^{5}^4, A);
--Next we define a general ideal in J
B = random(PP2^{-9,5:-8}, PP2^{4:-10,-9});
--This line is necessary to ensure the map is minimal
B = B - sub(B, {x=>0, y=>0, z=>0});

I = fittingIdeal(1, coker B);
--The result should be 0
HH^0(M ** sheaf module I)
\end{verbatim}

\medskip\noindent
    Running the code above in the Institute of Mathematics' computer cluster at UNAM took 62830 seconds. The cases $s<4$ are considerably faster.

\subsection*{Case $s=5$.}
In this case, instead of computing the cohomology groups of $M\otimes \mathcal{I}_Z$, we compute the cohomology of $M\otimes V$, where $V$ is the bundle \eqref{eq: esp res} in Lemma \ref{InterpolationTan}.
\begin{verbatim}
PP2 = QQ[x, y, z];

B2 = flatten entries basis(2, PP2);
C4 = for i in 0..23 list (random(PP2^{2}, PP2^{0}))_0_0;
A = matrix for i in 0..23 list
	for j in 0..4 list(
		if j == 4 then C4_i
		else if j == floor(i/6) then B2_(i-6*j)
		else 0
	);
M = sheaf coker map(PP2^{9}^24, PP2^{7}^5, A);

V = sheaf ker random(PP2^{-11}^2, PP2^{-12}^5);

time HH^1(M ** V)
\end{verbatim}

\end{document}